\documentclass[10pt]{amsart} \usepackage{amscd}
\usepackage{amssymb} \usepackage{amsmath} \usepackage{latexsym} \newtheorem{theorem}{Theorem}
\usepackage[all,cmtip]{xy}
\input xy
\xyoption{all}
\newtheorem{lemma}{Lemma}
\newtheorem{corollary}[theorem]{Corollary}
\newtheorem{proposition}[theorem]{Proposition}
\newtheorem{problem}{Problem}
\newtheorem{question}[problem]{Question}
\theoremstyle{definition}

\theoremstyle{definition}
\newtheorem{remark}{Remark}
\theoremstyle{definition}
\newtheorem{notation}{Notation}
\usepackage{amsfonts} 
\newcommand{\field}[1]{\mathbb{#1}} \newcommand{\Q}{\field{Q}}
\newcommand{\R}{\field{R}}  \newcommand{\Z}{\field{Z}}
\newcommand{\C}{\field{C}}  

\newcommand{\SL}{{\rm SL}}
\newcommand{\Sp}{{\rm Sp}}

\newcommand{\rank}{{\operatorname{rank}}}
   
\newcommand{\D}{\Delta}  
\newcommand{\g}{\gamma} \newcommand{\G}{\Gamma}
\renewcommand{\l}{\lambda}  
   
\newcommand{\s}{\sigma}    
\newcommand{\fg}{\mathfrak   g}   
   \newcommand{\fh}{\mathfrak  h}
    \newcommand{\bs}{\backslash}
\newcommand{\ra}{\rightarrow}
\begin{document}
\title[Discrete Linear Groups  containing Arithmetic Groups]{Discrete Linear Groups containing Arithmetic Groups}
\author{Indira Chatterji and T. N. Venkataramana}
\email{indira.chatterji@math.cnrs.fr, tn.venkataramana@icts.res.in}
\address{Indira  Chatterji, Universit\'e C\^ote d'Azur, 06000 Nice, FRANCE. 
Partially  supported by  NSF grant  DMS 0644613. \vskip 5mm    
T.  N. Venkataramana, Department  of Mathematics, International Centre for Theoretical Sciences, Bangalore- 560 089, INDIA. } \date{} 

\begin{abstract}
If $H$ is a simple real algebraic subgroup of real rank at least two in a simple real algebraic group $G$, we prove, 
in a substantial number of cases, that a Zariski dense discrete subgroup of $G$ containing a lattice in $H$ is a lattice 
in $G$. For example, we show that any Zariski dense discrete subgroup of $\SL_n(\R)$ ($n\geq 4$) which contains $\SL_3(\Z)$ (in the top left hand corner) is commensurable with a conjugate of $\SL_n(\Z)$. \\ 

In contrast, when the groups $G$ and $H$ are of real rank one, there are lattices $\Delta$ in a real rank one group $H$ embedded in a larger real rank one group $G$ and that extends to a Zariski dense discrete subgroup $\Gamma $ of $G$ of {\bf infinite co-volume}.
\end{abstract}

\maketitle

\tableofcontents

\section{Introduction}
In this paper, we study special cases of the following problem raised by Madhav Nori:

\begin{problem}[Nori, 1983] If $H$ is a real algebraic subgroup of a real semi-simple algebraic 
group $G$, find sufficient conditions on $H$ and $G$ such that any Zariski dense discrete 
subgroup $\Gamma $ of $G$ which intersects $H$ in a lattice in $H$, is itself a lattice in $G$.
\end{problem}

Until recently, we did not know any example when the larger discrete subgroup $\G$ was not a lattice, but where the smaller group $H$ is a simple Lie group and has real rank strictly greater than one (then the 
larger group $G$ is also of real rank at least two). The recent work of Danciger, Gu\'eritaud and Kassel \cite{DGK} provides non-lattice discrete subgroups which contain smaller higher rank lattices. However, there do exist examples (as in the present paper) where the larger $\G$ is forced to be a lattice: the goal of the present paper (written in 2009) is to study the following question, related to Nori's problem. 

\begin{question}\label{nori} If a Zariski dense discrete subgroup $\Gamma$ of
a simple non-compact Lie group $G$ intersects a simple non-compact Lie
subgroup $H$- of real rank at least two-in a lattice, is the larger discrete group $\Gamma$ a
lattice in the larger Lie group $G$?
\end{question}

Analogous questions have been considered before: (e.g. \cite{F-K}, Bass-Lubotzky \cite{B-L}, Hee Oh \cite{O}, and \cite{V2}).\\ 

Madhav Nori  first raised the question (to the second named author of the present paper) whether these larger discrete groups containing higher rank lattices have to be lattices themselves (i.e. arithmetic groups, in view of Margulis' Arithmeticity Theorem). We will in fact show several examples of pairs of groups $(H,G)$, such
that $\Gamma $ does turn out to be a lattice, giving examples where the answer to Question \ref{nori} is \emph{positive} (see Theorem \ref{spllinear} and Corollary \ref{SLn}).  As we already mentioned, however,   the recent results of \cite{DGK} show that there are many examples where the answer is \emph{negative}. It is not clear to us, what a reasonable conjecture  on the necessary and sufficient condition on the pair $(G,H)$ should be, so that  Question \ref{nori} has an  affirmative answer. \\

We prove a general result on super-rigidity of certain discrete groups, from which we can extract examples of $(G,H)$ with the answer "yes" to Question \ref{nori}. Here are two consequences: consider the ``top left hand corner'' embedding
\[\SL_k< \SL_n, ~(k\geq 3).\] 
The embedding is as follows: an $\SL_k$ matrix $M$ is thought of as an $n\times n$ matrix $M'$ such
that the first $k\times k$ entries of $M'$ are the same as those of $M$, the last $(n-k)\times (n-k)$ entries of $M'$ are those of the identity $(n-k)\times (n-k)$ matrix, and all other entries of $M'$ are zero.

\begin{theorem}\label{spllinear} Suppose that $\SL_3$ is embedded in
$\SL_n$ (in the ``top left hand corner'') as above. Suppose that
$\Gamma$ is a Zariski dense discrete subgroup of $\SL_n(\R)$ whose
intersection with $\SL_3(\R)$ is a subgroup of $\SL_3(\Z)$ of finite
index. Then, $\Gamma $ is commensurate to a conjugate of $\SL_n(\Z)$,
and is hence a lattice in $\SL_n(\R)$.
\end{theorem}

Here we recall that two groups are called \emph{commensurate} when their intersection has finite index in each of them. We may similarly embed $\Sp_k\subseteq \Sp_g$ in the ``top left hand corner'', where  $\Sp_g$ is  the symplectic group  of $2g\times 2g$-matrices  preserving the non-degenerate symplectic form $J_g={\rm diag}(J_2,\dots,J_2)$ in
$2g$-variables, where $J_2=\left(\begin{array}{cc}0&1 \\ 1&0\end{array}\right)$. Denote by $\Sp_g(\Z)$ the integral symplectic group.

\begin{theorem}\label{symplectic} If $\Gamma $ is a Zariski dense
discrete subgroup of $\Sp_g(\R)$ whose intersection with $\Sp_2(\R)<Sp_g(\R)$ (in the ``top left hand corner'') is
commensurate to  $\Sp_2(\Z)$ then a conjugate of  $\Gamma $ is
commensurate with $\Sp_g(\Z)$ and hence $\Gamma $ is a lattice in
$\Sp_g(\R)$.
\end{theorem}

It seems likely that $Sp_2 < Sp_g$ may be replaced by $SO(2,3) < SO(p,q)$ with $p \geq 2 \quad and \quad q \geq 3$, but we have not verified the details. \\

In the above examples, the group $H$ has very small dimension when compared to that of the ambient group $G$; nevertheless, these furnish examples where Question \ref{nori} has a positive answer. This is the reason we have singled them out. \\

The proof of Theorem \ref{spllinear} depends (as does the proof of Theorem \ref{symplectic}) on a general super-rigidity theorem for discrete subgroups $\Gamma$ which contain a ``large'' enough higher
rank lattice. More precisely, our main result is the following. \\ 

\begin{notation}\label{notation:Lie} Let $P$ be a minimal real parabolic subgroup of the simple Lie group $G$ and denote by $N$ the unipotent in a maximal real split torus of $P$.  Let $A$ be the connected component of identity in $S$. Denote by $P_0$ the subgroup $AN$ of $P$ (if $G$ is split, then $P_0$ is the identity component of $P$; in general we have replaced the minimal parabolic subgroup $P$ by a subgroup $P_0$ which has no compact factors such that $P/P_0$ is compact). Let $K$ be a maximal compact subgroup of $G$.  We have the Iwasawa decomposition $G=KP_0=KAN$. 
\end{notation}

\begin{theorem}[Main result]\label{theo:main} Let $H$ be a semi-simple subgroup of a semi-simple Lie group $G$ with $\R-\rank(H)\geq 2$ such that the normal subgroup of $G$ generated by $H$ is all of $G$. Let
$\Gamma $ be a Zariski dense subgroup of $G$ whose intersection with
$H$ is an irreducible lattice in $H$. Let $G=KAN$ be the Iwasawa
decomposition of $G$ and $P_0=AN$. If the isotropy of $H$ acting on $G/P_0$ is positive dimensional and non-compact at every point of $G/P_0$, then $\Gamma $ is a super-rigid subgroup of $G$.
\end{theorem}

An earlier version of Theorem \ref{theo:main} assumed $G$ to be simple. It was pointed out to us by Yehuda Shalom that the proof works for $G$ semi-simple as well. \\

Here we recall that a lattice $\Delta$ in a semisimple Lie group $H$ is \emph{irreducible} if for every non-central normal subgroup $N\lhd H$, $\Delta$ is dense when projected onto $H/N$.  The conditions of Theorem \ref{theo:main} are satisfied if the dimension of $H$ is sufficiently large (for example, if $\dim (K)<\dim (H)$). We then get the following as a corollary.

\begin{theorem} \label{ranktwo} Let $\Gamma $ be a Zariski dense discrete subgroup of a simple Lie group $G$ which intersects a semi-simple subgroup $H$ of $G$ (with $\R-\rank(H)\geq 2$) in an irreducible lattice. Let $K$ be a maximal compact subgroup of $G$ and assume that $\dim (H)> \dim (K)$. Then $\Gamma$ is a super-rigid subgroup of $G$.
\end{theorem}

Here, \emph{super-rigid} is in the sense of Margulis \cite{M}. That is, all linear representations - satisfying some mild conditions - of the group $\Gamma$ virtually extend to (i.e. coincide on a finite index
subgroup of $\Gamma$ with) a linear representation of the ambient group $G$.\\ 

For instance, as a consequence of Theorem \ref{ranktwo} we obtain the following.

\begin{corollary} \label{SLn} If $n\geq 4$ and $\Gamma$ is a Zariski dense discrete subgroup of $\SL_n(\R)$
which intersects any $\SL_{n-1}(\R)<SL_{n}(\R)$ in a finite index subgroup of $\SL_{n-1}(\Z)$, then a conjugate of $\Gamma$ in $\SL_n(\R)$ is commensurate to $\SL_n(\Z)$. \\

In the above we may replace $n-1$ by any integer $k>\frac{n}{{\sqrt 2}}$ {\rm (}and $k\leq n-1${\rm )}. 
\end{corollary}

Analogously, we prove

\begin{corollary} \label{Spn} If $g\geq 3$, every Zariski  dense discrete
subgroup of $\Sp_g(\R)$ which contains a finite index subgroup of
$\Sp_{g-1}(\Z)$ is commensurable to a conjugate of $\Sp_g(\Z)$.
\end{corollary}

The proof of Theorem \ref{theo:main} runs as follows.  \\ 

We adapt Margulis' proof of super-rigidity to our situation; the proof of Margulis uses crucially that a lattice 
$\Gamma $ acts ergodically on $G/S$ for any non-compact subgroup $S$ of $G$, whereas we do not
have the ergodicity available to us. We use instead the fact that the representation of the discrete group 
$\Gamma $ is rational on the smaller group $\Delta$.  \\ 

Given a representation $\rho $ of the group $\Gamma $ on a vector space over a local field $k'$, we use a construction of Furstenberg to obtain a $\Gamma $-equivariant measurable map $\phi $ from $G/P_0$
into the space ${\mathcal  P}$ of probability measures on the projective space of the vector space. Using the fact that the isotropy
subgroup of $H$ at any point in $G/P_0$ is non-compact, we deduce that on $H$-orbits, the map $\phi$ is rational, and by pasting together the rationality of $\phi$ on many such orbits, we deduce the rationality
of the representation $\rho$.  \\ 

Theorem \ref{ranktwo} implies Corollary \ref{SLn} as follows: the
super-rigidity of $\Gamma $ in Corollary \ref{SLn} implies, as in
\cite{M}, that $\Gamma $ is a {\it subgroup} of an arithmetic subgroup
$\Gamma _0$ of $\SL_n(\R)$. It follows that the $\Q$-form of $\SL_n(\R)$
associated to this arithmetic group has $\Q$-rank greater than
$n/2$. The classification of the $\Q$-forms of $\SL_n(\R)$ then implies
that the $\Q$-form must be  $\SL_n(\Q)$ and that $\Gamma $ is, up to conjuguation,
commensurate to a subgroup of $\SL_n(\Z)$. Since $\Gamma $ is Zariski
dense and virtually contains $\SL_{n-1}(\Z)$, it follows from
\cite{V1} Corollary (3.8) that $\Gamma $ is commensurate to $\SL_n(\Z)$. \\ 

We now give a (not exhaustive) list of pairs $(H,G)$ which
satisfy the condition ($\dim (H)>\dim(K)$) of Theorem \ref{ranktwo}.
\begin{corollary} If $(H,G)$ is one of the pairs
\begin{enumerate}
    \item $H=\Sp_g< \SL_{2g}$ with $g\geq 2$, 
\item $H=\SL_p\times \SL_p< G=\SL_{2p}$ with $p\geq 3$
\item $H=\Sp_a\times \Sp_a< G=\Sp_{2a}$ with $1\leq a$ 
\end{enumerate}
then any Zariski dense discrete subgroup $\Gamma $ of $G(\R)$ whose
intersection with $H(\R)$ is an irreducible lattice, is super-rigid in
$G(\R)$.
\end{corollary} 

There are examples of pairs $(H,G)$ satisfying the condition of
Theorem  \ref{theo:main}  which  are  not covered  by Corollary
\ref{ranktwo}. In the cases of the following corollary, it is easy to
check that $\dim (H)=\dim(K)$ and that the isotropy of $H$ at any
generic point of $G/P$ is a non-compact Cartan subgroup of $H$. When
$G=H(\C)$, we view $G$ as the group of {\bf real} points of a complex
algebraic group, and Zariski density of a subgroup $\Gamma \subseteq G$
is taken to mean that $\Gamma $ is Zariski dense in $G(\C)=H(\C)\times
H(\C)$.
\begin{corollary} Let $H$ be a real simple algebraic group defined over $\R$ with 
$\R-\rank(H)\geq 2$ embedded in the complex group $G=H(\C)$. If $H$ has no
compact Cartan subgroup, then every Zariski dense discrete subgroup
$\Gamma $ of $G$ which intersects $H$ in a lattice, is super-rigid in
$G$.
\end{corollary}
Notice that Theorem \ref{ranktwo} and Theorem \ref{theo:main} allow
us to deduce that if $\Gamma $ is as in Theorem \ref{ranktwo} or
Theorem \ref{theo:main}, then $\Gamma$  is a subgroup of an
arithmetic group in $G$ (see Theorem \ref{arithmetic}), reducing
Nori's question to the following apparently simpler one:
\begin{question} \label{nori'} If a Zariski dense subgroup $\Gamma $
of a lattice in a simple non-compact Lie group $G$ contains a higher
rank lattice of a smaller group, is $\Gamma $ itself a lattice in $G$?
\end{question}

We now know, from the results of \cite{DGK}, that the answer to Question \ref{nori'} is also negative in general. \\

We now briefly describe the proof of Theorem \ref{spllinear}. If Theorem
\ref{ranktwo} is to be applied directly, then the dimension of the
maximal compact of $\SL_n(\R)$ must be less than the dimension of
$\SL_3(\R)$, which can only happen if $n=4$; instead, what we will do,
is to show that the group generated by $\SL_3(\Z)$ (in the top left
hand  corner of  $\SL_n(\R)$ as  in the  statement  of Theorem
\ref{spllinear}) and a conjugate of a unipotent root subgroup of
$\SL_3(\Z)$ by a generic element of $\Gamma $,
(modulo its radical), is a Zariski dense {\it discrete} subgroup of
$\SL_4(\R)$. By applying Corollary \ref{SLn} for the pair $\SL_3(\R)$
and $\SL_4(\R)$, we see that the Zariski dense discrete subgroup
$\Gamma $  of $\SL_n(\R)$  contains, virtually, a  conjugate of
$\SL_4(\Z)$.  \\ 

We can apply the same procedure to $\SL_4(\R)$ instead of $\SL_3(\R)$ and obtain
$\SL_5(\Z)$ as a subgroup of $\Gamma $, etc, and finally obtain that
$\Gamma $ virtually contains a conjugate of $\SL_n(\Z)$. This proves
Theorem \ref{spllinear}. The proof of Theorem \ref{symplectic} is
similar: use Corollary \ref{Spn} in place of Corollary \ref{SLn}.  \\ 

When $H$ has real rank one, the answer to the counterpart of Question
\ref{nori} is in the negative: \\ 

For $G$ of rank two or higher, work of D. Johnson and J. Millson in \cite{J-M}
produces a Zariski dense subgroup $\Gamma$ of $\SL_n(\R)$, isomorphic to a
lattice in $SO(n-1,1)$ (hence of infinite co-volume in $\SL_n(\R)$)
and intersecting a subgroup $H$ isomorphic to
$SO(n-2,1)$ in a lattice $\Delta$
(see Remarque 1.3 of \cite{Ben1} and Corollaire 2.10 of \cite{Ben2}). \\ 

Even if $G$ has rank one, given a proper subgroup $H< G$ and a
lattice $\Delta < H$, one can always produce a Zariski dense
discrete subgroup $\Gamma$ in $G$ which is not a lattice in $G$, and
whose intersection with $H$ is a subgroup of finite index in $\Delta$
(Theorem \ref{rankone}). The method is essentially that of Fricke and
Klein \cite{F-K} who produce, starting from Fuchsian groups, Kleinian
groups of infinite co-volume, by using a ``ping-pong'' argument.  \\ 

At the suggestion of the referee, we make some remarks on extension of
these methods to groups over non-archimedean fields. The lattices in
such groups are known to be co-compact and therefore do not contain
unipotent elements.  Therefore, the algebraic methods of of the
present paper, proving that certain subgroups are lattices, by
exhibiting many  unipotent elements, do not  work. However, if the subgroup
$H$ operates on $G/P_0$ with positive dimensional isotropy groups,
then it can be shown by similar methods that a Zariski dense discrete
subgroup of $G$ which intersects $H$ in an irreducible lattice, is a
super-rigid group.  \\ 

We end with some remarks on Zariski closures. If $G$ is a connected algebraic group defined over $\R$, then it is known that $G(\R)$ is Zariski dense in $G$; hence, if $\Gamma < G(\R)$ is a subgroup, then the Zariski closure of $\Gamma$ in $G$ has the property that the smallest real algebraic group containing $\Gamma $ has finite index in
the real points of the complex Zariski closure. For this reason, we abuse notation a little and  refer to the set of 
 real points of the Zariski closure of the group $\Gamma $ as the real Zariski closure of $\Gamma$. \\

\subsection*{Note} This paper was written in 2009, but didn't get published;  The work   of  \cite{DGK} in 2024 renewed interest in the  results contained in the present work. In view of the results in \cite{DGK} answering No to Question \ref{nori} in such generality, one could wonder what the right conjecture is.  \\

\section{Preliminaries on measurable maps}

The aim of this section is to recall a few well known facts used in the proof of Theorem \ref{ranktwo} and prove Proposition \ref{rational} (this is a key fact in the proof of the main theorem).

\subsection{A Mild Generalisation of the Howe-Moore Theorem.}  

An important ingredient in the proof is the following easy generalisation (Lemma \ref{mautner} below)  of the Ergodicity Theorem of Moore (see \cite{Z}, Theorem (2.2.6)). This version is only slightly different from Theorem (2.2.6) in Zimmer's book (\cite{Z}); this is to take care of additional compact factors.  \\

We give the proof for the sake of completeness. We note that in the statement of the Lemma \ref{mautner} below, 
we may replace $\tilde{H}(\Z)$ by any irreducible lattice $\G$ in a real group $H$ such that $\G$ maps densely 
onto the maximal compact quotient of $H$; the arithmetic structure is not really used in the proof. \\

The almost $\Q$-simplicity of $\tilde H$ implies that there exists a connected absolutely almost simple simply connected group $H_0$ defined over a number field $K$, such that $\tilde H=R_{K/\Q}(H_0)$,
where $R_{K/\Q}(H_0)$ is the Weil restriction of $H_0$ from $K$ to $\Q$. Consequently,
\[\tilde  H  (\R)=H_0(K\otimes  \R)=\prod _{\alpha  \in  \infty} H_0(K_{\alpha}),\] where $\infty$ denote the set of equivalence
classes of archimedean embeddings of $K$. Let $A\subseteq\infty$ denote the archimedean embeddings $\alpha$ such that $H_0(K_{\alpha})$ is non-compact.  Then, $\tilde H= H^*\times H^u$ where $H^u:=\prod
_{\alpha \in \infty \setminus A} H_0(K_{\alpha })$ is a compact group, and $H^*:= ~\prod _{\alpha \in A} H_0(K_{\alpha})$ is a semi-simple group without compact factors.  \\

 As in the case of the Moore Ergodicity Theorem of \cite{Z}, we first show the vanishing of matrix coefficients.

\begin{lemma}\label{howemoore} Let  $\tilde{H}$ be an almost $\Q$-simple, simply connected algebraic group with $\R-\rank(\tilde{H})\geq 1$ and $\Delta < \tilde{H}(\Z)$ an arithmetic subgroup.  Suppose that $\pi$ is a unitary representation of $\tilde H(\R)$ on a Hilbert space, such that for any {\bf non-compact}  simple factor $H_{\alpha }$ of $H^*$, the space $\pi ^{H_{\alpha }}$ of vectors of $\pi $ invariant under the subgroup $H_{\alpha}$ is zero. Then, the space $\pi ^{S}$ of vectors in $\pi $ invariant under the non-compact subgroup $S<\tilde H(\R)$ is also zero.
\end{lemma}

\begin{proof} By the Howe-Moore theorem, (see \cite{Z}, Theorem (2.2.20)), for every pair of vectors $v,w\in \pi$ the matrix coefficient  $<g^*v,w>$ tends to zero as $g^*$ tends to infinity in the group $H^*$. \\

Suppose $g(n)=(g^*(n),g_u(n))  \in \tilde{H}(R)=H^* \times H_u$ is a sequence which tends to infinity. Since $H_u$ is compact, we may assume, by passing to a subsequence, that $g_u(n)$ tends to an element $k\in H_u$. 
Hence $g_u(n)v$ tends to some vector $v'\in \pi$. Then $g^*(n)$ also tends to infinity in $H^*$. Moreover, (write $g=g(n), g^*=g^*(n),g_u=g_u(n)$ for short): 
\[ <gv,w>=<g^*v',w>+ <g^*(g_uv-v'),w>.\]
The Cauchy Schwarz estimate and the unitarity of $g^*$ show that 
\[\mid <g^*(g_uv-v'),w> \mid \leq \mid g_uv-v'\mid \mid w\mid\] and hence tends to zero. By the Howe-Moore theorem, $<g^*v',w>$ tends to zero. Hence the above implies that $<gv,w>$ tends to zero as $g= g(n)$ tends to infinity in $\tilde{H}(\R)$. \\

In particular, no non-compact subgroup of $\tilde H(\R)$ can fix a non-zero vector in $\pi$. 
\end{proof}

\begin{lemma}\label{mautner} Let $\tilde{H}$ be an almost $\Q$-simple, simply connected algebraic group with 
$\R-\rank(\tilde{H})\geq 1$ and $\Delta <\tilde{H}(\Z)$ an irreducible lattice so that all the projections to the compact factors are dense. For any closed non-compact subgroup $S<\tilde{H}(\R)$, the group $\Delta$ acts ergodically on the quotient $\tilde{H}(\R)/S$.
\end{lemma}

\begin{proof} Let $V_0=L^2(\Delta \bs \tilde H(\R))$ be the space of square integrable functions on $\Delta \bs \tilde H(\R)$; the latter space has finite volume (Theorem 1 of \cite{BHC}), 
and hence contains the space of constant functions.  For any simple factor $H_{\alpha }$ of $H^*< \tilde H(\R)$, the space $V_0^{H_{\alpha }}$ is just the space of constants, by strong approximation (see \cite{M}, 
Chapter (II), Theorem (6.7); in the notation of \cite{M}, we may take $B=\{{\alpha \}}$ to be a singleton). Consequently, if $\pi $ denotes the space of functions in $V_0$ orthogonal to the constant functions,  then $\pi  $  satisfies the  assumptions of  Lemma \ref{howemoore}. Therefore, by Lemma \ref{howemoore}, $\pi ^S=0$, and
hence the only functions on $\Delta \bs \tilde H(\R)$ invariant under the non-compact  group $S$  are constants.  This  proves Lemma~\ref{mautner}.
\end{proof}

We now record a statement which will be used in the proof of Proposition \ref{rational}. We thank the referee for pointing out the simple proof (and the correct formulation) of the following lemma.

\begin{lemma}\label{useful} Let $\tilde{H}$ be a $\Q$-simple, simply connected algebraic group with $\R$-rank$(\tilde{H})\geq 1$ and $\Delta <\tilde{H}(\Z)$ an arithmetic subgroup.  Suppose that $s\in \tilde{H}(\R)$ generates an  infinite discrete subgroup and let
$\tau:s^{\Z}\ra \Z$ be an isomorphism. Then, there is no $s^{\Z}$-equivariant Borel measurable map
\[\phi^*:\Delta \bs \tilde{H}(\R)\ra \Z.\]
\end{lemma}

\begin{proof} The image (push-out) of the Haar measure on $\tilde{H}(\R)/\Delta$ under an  $s^{\Z}$-equivariant Borel measurable map $\phi^*$ gives a finite $\Z$-invariant measure on $\Z$, which cannot exist.
\end{proof}

\subsection{The Margulis Super-Rigidity Theorem.} The following is a version of the Margulis super-rigidity theorem, except that the Zariski closure of the image $\rho (\Delta )$ is not assumed to be an absolutely simple group and that $\rho (\Delta )$ is not assumed to have non-compact closure in $G'(k')$ if $k'$ is archimedean.

\begin{theorem}[Margulis]\label{super-rigid} Let $\tilde{H}$ be a $\Q$-simple simply connected algebraic group defined over $\Q$ of $\R$-rank~$(\tilde  H)\geq 2$,  let $\Delta < \tilde{H}(\Z)$ be a subgroup of finite index and $\rho : \Delta \ra G'(k')$ a homomorphism into a linear algebraic group $G'$ over a local field $k'$ of characteristic zero. \\

\begin{enumerate}
\item If $k'$ is archimedean, then the map $\rho $ coincides, on a subgroup of finite
index, with a representation $\tilde \rho: \tilde H(\R) \ra G'(k')$. \\

\item If the local
field $k'$ is non-archimedean, then $\rho (\Delta )$ is contained in a
compact subgroup of $G'(k')$.
\end{enumerate}
\end{theorem}

\begin{proof} The usual statement of Margulis' super-rigidity says that if $\rho$ is a homomorphism of an irreducible lattice $\Gamma $ in a real semi-simple linear Lie group $H$ without compact factors, and if $G'$ is an absolutely simple group defined over an archimedean local field $k'$, then a representation from $\Gamma $ into $G'(k')$ with Zariski dense image,  extends to a smooth
representation of $H$ into $G'(k')$ (see \cite{M}, Chapter VIII, Theorem (C) ). However, if $\tilde H$ is the group of
real points of a $\Q$-simple simply connected algebraic group, then $\tilde H(\R)$ may have compact factors and 
hence $\tilde{H}(\Z)$ (or a finite index subgroup of $\tilde{H}(\Z)$) may have representations whose Zariski closures
have compact factors. \\

Even so, representations of $\tilde{H}(\Z)$ do extend to $\tilde H(\R)$ under the hypotheses of Theorem \ref{super-rigid}. 
First of all, the Zariski closure $\mathcal H$ of any representation of a higher rank lattice is semi-simple (\cite{M} Theorem (3.10), Chapter VIII).
By passing to a finite index subgroup of $\rho$ we may assume that this Zariski closure $\mathcal H$  is connected and semi-simple. We may write $\mathcal H =H_1\cdots H_r$ as an almost direct product with each $H_i$ absolutely almost simple over an archimedean local field $k_i$ (which is $\R$ or $\C$). Hence the representation $\rho$ is of the form $\rho _1 \cdots \rho _r$, with each $\rho _i :\D \ra H_i$ with Zariski dense image. \\

The $\Q$-simple group $\tilde{H}$ is obtained as the Weil restriction of scalars from $K$ to $\Q$ of an absolutely almost simple simply connected group $H_0$ defined over a number field $K$. Namely, $\tilde{H}=R_{K/\Q}H_0$ (for the Weil restriction of scalar we refer to \cite{Z} Chapter 6, p.115). Then, $H(\R)=\prod _{\s \in K_\infty} H_0(K_\s)$ where $K_\infty$ is the set of inequivalent archimedean completions of the number field $K$. \\

Then, by Chapter VIII, Theorem (3.6), part (ii), (a) of \cite{M}, the representation $\rho _i$ coincides (up to a homomorphism of $\D$ into the centre) with an algebraic representation of $H_0$ into $H_i$ defined over the field $k_i$ and a homomorphism $\s_i: K \ra k_i$. 
But any such  archimedean embedding $\s_i$ is simply the inclusion of $K$ into $K_s$ (followed by a continuous embedding of fields $K_s \ra k_i$)  and hence the map $\rho _i :\D \ra H_i(k_i)$ is the composite of the inclusion of  $H_0(O_K)$ into $H_0(K_s)$ followed by an algebraic homomorphism $H_0 \ra H_i$ defined over $k_i$. 
In other words, $\rho _i$ is up to centre, an algebraic homomorphism of $H_0(K_s)$  into $H_i(k_i)$. Since this centre is finite, we may pass to a further subgroup of finite index in $\D$ to ensure that this homomorphism $\rho $ restricted to the finite index subgroup, coincides with  an algebraic representation of $\tilde H(\R)$.

\end{proof}

We can now prove the main result (Proposition \ref{rational}) of this section. In the proposition, when we talk of an algebraic subgroup $J$ of $G'(k')$ where $k'$ is an archimedean local field, we mean that $J$ is an algebraic $\R$-subgroup of the \emph{real}  group $R_{k'/\R}(G')$ obtained from $G'$ by the Weil restriction of scalars. 

\begin{proposition}\label{rational} Let  $\tilde{H}$ be an almost $\Q$-simple,   simply   connected   algebraic   
group   with $\R-\rank(\tilde{H})\geq 2$ and $\Delta <\tilde{H}(\Z)$ an arithmetic  subgroup.  Let  $\rho  :\Delta \ra  G'(k')$  be  a representation, where $G'$ is a linear algebraic group over a local field  $k'$  of characteristic  zero.  Suppose that  $S<\tilde{H}(\R)$ a closed non-compact subgroup and let $J<G'(k')$ be an algebraic subgroup. Then any Borel measurable and $\Delta$-equivariant map $\phi :\tilde H(\R)/S\ra G'(k')/J$ coincides with a rational
map almost everywhere on $\tilde{H}(\R)$. More precisely:\\ 


If $k'$ is an archimedean local field, there exist a homomorphism
$\tilde \rho : \tilde H(\R) \ra G'(k')$ of real algebraic groups
defined over $\R$, and a point $p\in G'(k')/J$,   such that for almost all $h\in \tilde{H}(\R)$, the map $\phi
(h)$ and the map $h\mapsto \tilde \rho (h)(p)$  coincide (that is to say,  $\phi (h)=\tilde \rho (h)p$ for almost all $h\in \tilde H(\R)$ and coincides almost everywhere with  an $\R$-rational map of real varieties on $\tilde{H}(\R)$). \\ 

If $k'$ is a non-archimedean local field, then the map $\phi$ is constant a.e. on $\tilde{H}(\R)$. 
\end{proposition}

\begin{proof}  Suppose first that $k'$ is archimedean. Let $\Delta '< \Delta $ be a subgroup of
finite  index such  that  there exists  (according to  Theorem \ref{super-rigid} quoted above) a representation $\tilde \rho :\tilde
H(\R)\ra G'(k')$ which coincides with $\rho $ on $\Delta '$. Consider the map $\phi ^*(h)=\tilde \rho (h)^{-1}(\phi (h))$ from $\tilde
H(\R)$ into the quotient $G'(k')/J$. Then, for all $\delta \in \Delta '$,  almost all 
$h\in \tilde H(\R)$ and all $s\in S$, we have $\phi ^*(\delta h)= \phi ^*(h)$ and $\phi ^*(hs)=\tilde \rho (s)^{-1}(\phi ^*(h))$. That is,
the map $\phi ^*$ is $\Delta '$ invariant and $S$-equivariant for the action of $\tilde H(\R)$ on $G'(k')/J$ via the representation $\tilde
\rho $.  \\ 

The representation $\tilde \rho$ is algebraic; moreover, since $k'$ is
archimedean, by assumption the group $J$ is a real algebraic subgroup
of $G'(k')$, and hence the action of $\tilde H(\R)$ on $G'(k')/J$ is
smooth. Let $S_1$ denote the Zariski closure of the image $\tilde
\rho (S)$. The $S_1$-action on $G'(k')/J$ is smooth, hence the
quotient $S_1\bs G'(k')/J$ is countably separated. On the other hand,
by Lemma \ref{mautner}, the action of $S$ on $\Delta '\bs \tilde
H(\R)$ is ergodic. Hence, by Proposition (2.1.11) of \cite{Z}, the
image of $\phi ^*$ is essentially contained in an $S_1$-orbit
i.e. there exists a Borel set $E$ of measure zero in $\tilde H(\R)/S$,
such that the image under $\phi$ of the complement of $E$ is
contained in an $S_1$-orbit. \\

Since $S$ is a non-compact Lie group, $S$ contains an element $s$ of
infinite order which generates a discrete non-compact subgroup. We may
replace $S$ by the Zariski closure of the group generated by the element $s$ and $S_1$ 
by the Zariski closure of the image of $s^{\Z}$. Hence the $S_1$-orbit of the preceding paragraph is of the
form $S_1/S_2$ with $S_2$ an algebraic subgroup of the abelian group
$S_1$. \\

{\it Case 1:} Suppose that the inverse image $S'=S\cap \tilde \rho ^{-1}(S_2)$ is a
non-compact subgroup. By Lemma \ref{mautner}, the group $S'$ acts
ergodically on $\Delta \bs \tilde H(\R)$; therefore, the map $\phi ^*$
- being $S'$ invariant - is essentially  constant: $\phi ^*(\Delta \bs \tilde
H(\R))=\{p\}$ for some point $p\in G'(k')/J$. That is $\phi (h)=\tilde
\rho (h)(p)$ a.e. on $\tilde{H}(\R)$, and coincides with a  rational map a.e.\\

{\it Case 2:} Suppose that $S'$ is compact. Since $s^{\Z}$ generates a discrete
non-compact subgroup and $S'$ is compact, the image $s_1^{\Z}:= \tilde
\rho (s^{\Z})$ also generates a discrete non-compact subgroup in
$S_1/S_2$ ($\tilde \rho $ being an algebraic, hence continuous, map).
Hence $s_1^{\Z}$-orbits in $S_1/S_2$ are closed so that the space
$s_1^{\Z}\bs S_1/S_2$ is countably separated. By Lemma \ref{mautner},
$s^{\Z}$ acts ergodically on $\Delta \bs \tilde H(\R)$. By applying
Proposition (2.1.11) of \cite{Z} to the $s_1^{\Z}$-invariant map
$\overline{\phi ^*}: \Delta \bs \tilde H(\R)\ra s_1^{\Z}\bs S_1/S_2$,
we deduce that the image of $\phi ^*$ is essentially contained in an
orbit of $s_1^{\Z}$ in the quotient group $S_1/S_2$.  Then, by Lemma
\ref{useful}, it follows that $\phi ^*$ cannot exist and we are in Case 1. \\ 

If $k'$ is non-archimedean, then by Theorem \ref{super-rigid}, the
image $\rho (\Delta )$ is contained in a compact group $K$ which acts
smoothly on  $G'(k')/J$ so  that $K\bs G'(k')/J$  is countably
separated. The group $\Delta $ acts ergodically on $\tilde H(\R)/S$.
Hence, by Proposition  (2.1.11) \cite{Z}, the $S$-invariant map
$\overline{\phi} :\Delta \bs \tilde {H}(\R) \ra K\bs G'(k')/J$ is
essentially constant, and therefore the image of $\phi$ is essentially
contained in an orbit of $K$.    \\

Since $K$ is a compact subgroup of the $p$-adic group $G'(k')$, it has
a decreasing sequence of open subgroups $(K_n) _{(n\geq 1)}$ (of
finite index in $K$) such that the intersection $\cap _{n\geq 1}
K_n=\{1\}$ is trivial. Then $\Delta _n=\Delta \cap \rho ^{-1}(K_n)$
is of finite index in $\Delta$. By Theorem \ref{super-rigid} applied to
$\Delta _n$, it follows that the image of $\phi $ is contained in an
orbit of $K_n$ for each $n\geq 1$. But since the subgroups $K_n$'s
converge to the identity subgroup, it follows that the image of $\phi
$ is a singleton. That is, $\phi $ is constant on a conull subset of $\tilde{H}(\R)/S$.
\end{proof}

\subsection{Some Measure Theoretic Constructions} 

We now mention a consequence of Fubini's Theorem we will need in the proof of Theorem~\ref{ranktwo}. \\

\begin{notation} Suppose that $H$ is a locally compact Hausdorff
second countable topological group with a Haar measure $\mu$ and assume that
$(H,\mu )$  is $\sigma$-finite.  Suppose that $(X,\nu)$  is a
$\sigma$-finite measure space on which $H$ acts such that the action
$H\times X\ra X$ - denoted by $(h,x)\mapsto hx$ - is measurable and so
that for each $h\in H$, the map $x\mapsto hx$ on $X$ preserves the
measure class of $\nu $. Let $Z$ be a measure space and let $f:X\ra Z$ be
a measurable map. \\

If $\Sigma =\Sigma (X,Z)$ is the set of measurable maps from $X$ to $Z$, then $f\in \Sigma$ and $H$ acts on $\Sigma$ by left translations on $X$. We regard two maps $\phi, \psi \in \Sigma$ to be equal if they coincide almost everywhere on $X$.

\end{notation}

\begin{lemma}\label{almostall} Under the above notation, suppose that there exists a co-null subset $X_1\subseteq X$ such that for each $x\in X_1$, the map $h\mapsto f(hx)$ is constant on a co-null subset $H_x$ of $H$.  \\

Then, given $h\in H$, there exists a measurable subset $X_h$ which is co-null in $X_1$ and such that for all $x\in X_h$
\[f(hx)= f(x).\] 

 This equation says that $H$ lies in the isotropy of $f \in \Sigma (X,Z)$: for every $h\in H$, $hf=f$ in $\Sigma$. 
\end{lemma}

\begin{proof} First, let $F=\{(x,a,b)\in X_1\times H \times H: f(ax)=f(bx) \}$. Being a pullback of the diagonal in $Z\times Z$ under a measurable map, the set $F$ is measurable. Further, for each $x$ in the conull set $X_1$, the slice $\{x\}\times H _x \times H_x$ lies in $F$, and is conull in $\{x\}\times H\times H$, so the set $F$ is conull.\\

By Fubini (see \cite{Hal} Theorem A p.147) there exists $h_0\in H$ such that the intersection $E\times\{h_0\}:=F\cap (X\times H\times\{h_0\})$ is conull in $X\times H\times\{h_0\}$, so that $E$ is conull in $X\times H$.\\

Applying Fubini to $E$, there exists a conull subset $H_1\subseteq H$ such that for each $h\in H_1$ there is $X_h\subseteq X$ conull  such that $f(hx)=f(h_0x)$ for all $x\in X_h$. Thus, the two $Z$-valued functions $x\mapsto f(hx)$ and $x\mapsto f(h_0x)$ are equal a.e. on $X$ and so lie in the same equivalence class of $\Sigma$. \\

On the set $\Sigma =\Sigma (X,Z)$ of measurable maps (modulo the equivalence that equality a.e. implies equivalent), the group $H$ operates by left translations on $X$. By the preceding paragraph, $h f =h'f=h_0f$ for all $h,h' \in H_1$, so that $h^{-1}h'f=f$ lies in the isotropy of $f$.\\


Given $g\in H$, the set $H_1g$ is again conull and so intersects $H_1$; that is, there exist $h,h'\in H_1$ such that $h'=hg$ and hence \emph{every} $g\in H$ is of the form $h^{-1}h'$ for some $h,h'\in H_1$. Hence, by the preceding paragraph, $gf=f$ for all $g\in H$, proving the lemma. 
\end{proof}

We recall a well known result of Furstenberg we will need, commonly known as Furstenberg's Lemma but due to Zimmer in the form given below.

\begin{lemma}[Furstenberg, \cite{Z} Corollary (4.3.7) and Proposition (4.3.9)]\label{furstenberg} 
Suppose that $\Gamma $ is a closed subgroup of a locally compact topological group $G$ and that $P_0$ is
a closed amenable subgroup of $G$.  Let $X$ be a compact metric $\Gamma $-space. Then there exists a Borel measurable $\Gamma $-equivariant map $\phi$ from a conull subset of $G/P_0$ to $\mathcal{P}(X)$, the space of probability
measures on $X$.
\end{lemma}

\subsection{Some Lie Theoretic Results} 

We recall that $G=G(\R)$ is a connected simple Lie group and $\G< G$ a Zariski dense discrete subgroup and $H<G$ a semi-simple subgroup (we  denote  by $H$ the group of real points $H(\R)$). In this section, we suppose that $H$ is a semi-simple Lie subgroup of a {\bf simple} Lie group $G$ and refer to Notation \ref{notation:Lie} of the Introduction. 
Let $\fg$ and $\fh$ be the Lie algebras of $G$ and $H$ respectively and $X=G/P_0$. Denote by $^{\g}(H)$ the conjugate $\g H \g ^{-1}$. \\

\begin{lemma}  \label{G/Pgenerators} Let $G$  be a connected Lie group, $H$ a semisimple subgroup and $\Gamma<G$ a discrete Zariski dense subgroup. There exist finitely many elements $\g_1,\cdots, \g_k \in \G$ such that $G$ is generated by the $H_i=\!{\ }^{\g _i}(H)$ for $i=1,\dots,n$. Moreover, for {\bf every} $x\in G/P_0$, the map $H_k\times \cdots \times H_1 \ra G/P_0$ given by $(h_k, \cdots,h_1)\mapsto h_k\cdots h_1x$, is surjective.
\end{lemma} 

\begin{proof}  Let $U\subseteq G$ be a small symmetric (i.e. $U=U^{-1}$) open neighbourhood of the identity. Then the countable union $\cup _{m=1}^{\infty} U^m$ is an open connected subgroup, so that $G=\cup _{m=1}^{\infty} U^m$.  \\

Consider the sum  $W=\sum _{\g \in \G} {\ }^{\g }(\fh) \subseteq \fg$ of the subspaces $^{\g}(\fh)$. This is stable under the action of $\G$ under the adjoint representation of $G$ and by the Zariski density of $\G$, it is stable under the adjoint action of $G$; the simplicity of $G$ implies that $\fg$ is irreducible for the action of $G$ and hence $W=\fg$. Hence there exist finitely many elements $\g_i \in \G$ such that 
\[\fg=\sum _{i=1}^l (^{\g _i}(\fh)) .\] 
By the implicit function theorem, there exists a small symmetric open neighbourhood $U$ of identity in $G$ such that $U$ is contained in the product set 
\[\Pi={\ }^{\g _l}(H)\cdots {\ }^{\g _1}(H).\] Hence the group generated by the product $\Pi$ contains $U^m$ for all $m$ and it follows that it is $G$. \\

Now, $G/P_0=\cup _{m=1}^{\infty} X_m$ with $X_m=U^m(P_0)$ an increasing sequence of open sets. The compactness of $G/P_0$ implies the existence of some $m$  such that $G/P_0=U^m(P_0)=X_m$. Thus, given $x\in G/P_0$, $x\in U^m(P_0)$; equivalently, $P_0=U^m(x)$. Hence $G/P_0=U^{2m}(x)$ for all $x\in G/P_0$. \\

By taking $U$ small enough, we can assume that $U$ is contained in the $l$-fold product set $\Pi=\!{\ }^{\g _l}(H)\cdots ^{\g _l}(H)$. It follows that $G/P_0 \subseteq (^{\g _1}(H)\cdots ^{\g _l}(H))^{2m} (x)$ for {\it every} $x\in G/P_0$:  
$$G/P_0=U^{2m}(x) \subseteq \Pi ^{2m}(x) \subseteq G/P_0.$$
We  may take $k=2ml$ and take the $(\g _j)_{1\leq j \leq k}$ to lie among the elements $\g_1, \cdots, \g_l$ with possible repetitions. This proves the lemma. 
\end{proof}

\subsection{Invariance properties of a map associated to $\phi$.}

In addition to the assumptions preceding Lemma \ref{G/Pgenerators} assume that the subgroup $H$ operates with non-compact isotropies on $G/P_0$. Suppose $\G$ is a Zariski dense discrete subgroup of $G$ whose intersection with $H$ is an irreducible lattice. Thus we are under the assumptions of Theorem \ref{theo:main}. \\
\begin{lemma} \label{constantp} Let $H,G,\Gamma$ as in Theorem \ref{theo:main}. Suppose $Z$ is a countably separated topological space. Then, any $p: G/P_0 \ra Z$ a measurable map which is $\G$-invariant, is constant on a co-null subset of $G/P_0$.
\end{lemma}

\begin{proof} Consider the space $\Sigma $ of measurable maps from $G/P_0$ into $Z$. On $\Sigma$ the group $G$ operates on the left by translations on $G/P_0$. Thus, for an element $g\in G$, $gp=p$ means  that $p(g^{-1}x)=p(x)$ almost everywhere on $G/P_0$. The isotropy at the point $p$, of the action of $G$ on the set  $\Sigma $,  is a subgroup of $G$. The invariance of the map $p$ under $\G$ implies that the isotropy at $p$ contains $\G$. \\

Suppose $X\subseteq G/P_0$ is the conull subset on which $p$ is defined everywhere. \\

Write $H$ for $H(\R)$ as in Proposition \ref{rational}. For each $x\in X\subseteq G/P_0$ we have the map $p_x: h\mapsto p(hx)$ from $H$ into $Z$.  This map is defined almost everywhere on $H$. The invariance of $p$ under $\Gamma $ implies that the map $p_x$ is invariant under $\D$. The isotropy $H_x$ of $H$ at $x$ is non-compact. Since $Z$ is countably separated,  the ergodicity of the action of $\D$ on $H/H_x$, then implies that the map $p_x$ is constant almost everywhere (by Proposition  (2.1.11) of \cite{Z}). Then by Lemma \ref{almostall}, the isotropy of the map $p$ contains $H$.  Hence it contains the groups $\g H \g ^{-1}$ for every $\g \in \G$. By Lemma \ref{G/Pgenerators}, the isotropy is then all of $G$. \\

We then have, for each $g\in G$,  $p(gx)=p(x)$ a.e. on $X$. Recall that as a topological space $G=K\times P_0$. In particular, for each $k\in K$, we have $p(kx)=p(x)$ a.e. on $X$. Then by Lemma \ref{almostall}, there exists a co-null subset $Y\subseteq X$ such that for {\it each}  $x\in Y$, the equality $p(kx)=p(x)$ holds  on a co-null subset $B \subseteq K$. The image of $B$ under the {\it diffeomorphism}  $k\mapsto kx$ from $K$ onto $G/P_0$ is therefore co-null and $p$ is constant on this image. 
\end{proof}

Suppose we are under the hypotheses of Theorem \ref{theo:main}.  Thus we have a Zariski dense discrete subgroup $\G$ of the group $G$, and $G$ contains the semi-simple subgroup $H$ whose intersection with $\G$ is an irreducible lattice $\D$ 
in $H$. We have a homomorphism $\rho: \G \ra G'(k')$ with Zariski dense image where $G'$ is a centreless absolutely simple group defined over a local field $k'$ and $P'<G'$ a minimal parabolic $k'$-subgroup. Consider the action of $G'(k')$ on the space $\mathcal P$ of probability measures on the compact Hausdorff space $G'(k')/P'(k')$. Then the quotient $Z=\mathcal{P}/G'(k')$ is a topological space (which is $(T_0)$ but is not necessarily Hausdorff). Then $\G$ , which maps to $G'(k')$,  acts on $\mathcal P$. \\

\begin{proposition} \label{oneorbit} Under the hypotheses of Theorem \ref{theo:main},  there exists a $\G$-equivariant measurable map $\phi: G/P_0 \ra G'(k')/J$ for some closed subgroup $J< G'(k')$.
\end{proposition} 

\begin{proof} By the Furstenberg Lemma (Lemma \ref{furstenberg}), we have a $\G$-{\it equivariant} measurable map $\phi : G/P_0 \ra \mathcal{P}$. Composing this with the quotient map $\mathcal{P} \ra Z=\mathcal{P}/G'(k')$, we have a $\G$-{\it invariant}   measurable map $p: G/P_0 \ra Z$. By Corollary (3.2.17) of \cite{Z} ( see also (2.1.9) of \cite{Z}), the space $Z$ is countably separated. By Lemma \ref{constantp}  above, the map $p$ is a constant map, and hence the image of $\phi $ lies in an orbit of $G'(k')$ of a point $\mu $ in $\mathcal P$. \\

This orbit is  isomorphic to the quotient $G'(k')/J$ for some closed subgroup $J$ of $G'(k')$ since orbits in $\mathcal P$ are locally closed (Proposition (2.1.10) and Corollary (3.2.17) of \cite{Z}). 

\end{proof}

\subsection{An Equivariance Property of a Map Associated with $\phi $.} 

We assume that $G'$ is an absolutely simple algebraic group defined over \emph{archimedean} local field $k'$ and $J<G'(k')$ as in Proposition \ref{oneorbit}, so that we have a $\G$-equivariant measurable map $G/P \ra G'(k')/J$. Then  $\rho : H \ra G'(k')$ (the subgroup $H$ is such that it operates with non-compact isotropy at each point of $G/P_0$). By Proposition \ref{rational} , for each $x\in G/P_0$ the map $h\mapsto \tilde \rho (h)^{-1}\phi (hx)$  is constant a.e. on $H$. \\

Thus, the function $F: H \times X \ra G'/J$ given by $F(h,x)=\tilde \rho ^{-1}(h)\phi (hx)$ is such that for each $x\in G/P_0$, this map is constant a.e. on the slice $H\times x$. By Fubini, there exists a measurable function $\theta : X \ra G'/J$ such that on a conull set in $H\times X$ we have $F(h,x)=\theta (x)$. Consequently, for each $x\in X_1$,  a co-null subset in $X$, there exists a co-null set $H(x)\in H$ such that 
$\phi(hx)=\tilde \rho (h)\theta (x)$ for every $h\in H(x)$. Fix $h_1\in H(x)$, and assume $hh_1\in H(x)$; the set of such $h$ is also co-null. Then
$\tilde \rho (hh_1)^{-1}\phi (hh_1x) =\theta (x)$ and hence $\phi (hh_1x)=\tilde \rho (h) \tilde \rho (h_1)\theta (x)$. On the other hand since $h_1\in H(x)$, we have  $\tilde \rho (h_1)\theta (x)=\phi (h_1x)$. Therefore, $\phi (hh_1x)=\tilde \rho (h)\phi (h_1x)$. By Fubini, the set of elements $y=h_1x \in X_1$ with $h_1\in H(x)$ and $x\in X_1$ is co-null and hence $\phi (hy)=\tilde \rho (h)\phi (y)$ holds on a co-null set $Y \subseteq X$. \\  

Using the $\G$-equivariance of $\phi$ we get $\phi (h_ix)=\tilde \rho _i(h_i)\phi (x)$ for almost all $h_i\in H_i=\g_i H \g_i^{-1}$ and almost all $x\in X$ where $\tilde \rho _i (h_i)=\rho (\g_i) \tilde \rho (h) \tilde \rho (\g_i)^{-1}$, with $h_i=\g_ih\g_i^{-1}$. Then it follows from a repeated application of Fubini's theorem  that if $H_i=\g_iH\g_i^{-1}$, and $k\geq 1$, then  

\begin{lemma} \label {Hequivariance} For each $k$ -tuple $h$ in a conull set $W$ in $h=(h_k, \cdots, h_1)\in H_k\times H_{k-1}\times \cdots \times H_1$, there exists a  conull set $X'(h) \subseteq X$, such that for $h\in W_k$ and $x\in X(h)$,  we have 
\[\phi (h_kh_{k-1}\cdots h_1x)=\rho _k(h_k)\rho _{k-1} (h_{k-1})\cdots \rho _1(h_1) \phi (x).\]
\end{lemma}
\qed

\section{Proof of the super-rigidity result  (Theorem \ref{theo:main})}

We will now proceed to the proof of Theorem \ref{theo:main}.  We suppose that $H$ is a semi-simple Lie subgroup of a {\bf simple} Lie group $G$ and Notation \ref{notation:Lie} from the Introduction. We will treat the archimedean and non-archimedean cases separately.    \\

\subsection{The Non-Archimedean Case.}

\begin{theorem}\label{nonarch} Suppose that $\Gamma $ is a Zariski dense discrete subgroup of a simple Lie group $G$ which intersects a semi-simple Lie subgroup $H$ (of real rank at least two) of $G$ in an
irreducible lattice $\Delta$. Suppose that $H$ acts with non-compact isotropy at any point of $G/P_0$ {\rm(}or that $\dim (H)>\dim(K)$ for a maximal compact subgroup $K$ of $G${\rm )}.  Then, the group $\Gamma $ is
non-archimedean super-rigid {\rm (}that is, if $G'$ is an absolutely almost simple group defined over a non-archimedean local field $k'$ of characteristic zero, then every representation $\rho : \G \ra G'(k')$ has compact image{\rm)}.
\end{theorem}

\begin{proof} Suppose that $\rho :\Gamma \ra G'$ is a representation of $\Gamma $ into an absolutely simple (centreless) algebraic  group $G'$ defined over a {\bf non-archimedean} local field $k'$ of characteristic zero with Zariski dense image. \\

By Proposition \ref{oneorbit}, there exists a closed subgroup $J< G'(k')$ and $\G$-equivariant measurable map $\phi : G/P_0\ra G'(k')/J$. By Proposition \ref{rational}, for any $x\in G/P_0$, the map $\phi _x: h\mapsto \phi(hx)$ is constant a.e. in $H$ (we are therefore using  the non-archimedean version of Margulis' super-rigidity as in Theorem \ref{super-rigid} to conclude that the
image of the lattice in $H$ is contained in a compact subgroup in $G'(k')$). \\

 Let $\Sigma $ now denote the set of measurable maps from $G/P_0$ into $Z=G'(k')/J$. Then $\phi $ lies in $\Sigma$. The constancy of the map $\phi _x$ for almost all $x\in G/P_0$  then implies as in the proof of Lemma \ref{constantp}, that $h\phi =\phi$ for all $h\in H$. Thus $H$ leaves $\phi $ invariant. The equivariance of $\phi $ under $\G$ and invariance under $H$ then implies that for every $\g \in \G$, the conjugate group $^{\g} (H)$ also leaves $\phi $  invariant and hence the isotropy of $\phi$ contains the group generated by these conjugates. By Lemma \ref{G/Pgenerators}, this group is precisely $G$. Hence $G$ leaves $\phi $ fixed. By Lemma \ref{constantp}, the map $\phi $ is then constant. Hence $\rho (\Gamma )$ fixes a point $\mu $ in ${\mathcal P}$.    \\

But the isotropy subgroup $J$  in $G'(k')$ of a probability measure $\mu\in {\mathcal P}$ is (by Corollary (3.2.19) \cite{Z}) 
either compact, whence $\rho (\G)$ is contained in a compact group, or else the isotropy is contained in an algebraic group  $L < G'$ of strictly smaller dimension. The latter is impossible because we have assumed that $\rho (\G)$ is Zariski dense in $G'$. Therefore, $\rho (\G)$ lies in a compact subgroup of $G'(k')$.  \\

This means that $\Gamma $ is non-archimedean super-rigid in $G$, and proves Theorem \ref{nonarch}.

\end{proof}

\subsection{The Archimedean Case} 

\subsubsection{Preliminaries on the $G'$-action on $G'/J$.}  Recall that $\rho : \G \ra G'(k')$ where $k'$ is an {\it archimedean} local field and  $G'$ is an absolutely simple group of adjoint type over  $k'$ such that $G'(k')$ is not compact and $\rho (\G)$ is Zariski dense in the real algebraic group $G'$. By Furstenberg's lemma we have a $\G$-equivariant map $\phi : G/P_0\ra \mathcal{P}$.  By Proposition \ref{oneorbit},  $\phi $ lies in a $G'$ orbit of a measure $\mu \in \mathcal{P}$. The isotropy of $G'$ at $\mu$ is then  a closed subgroup $J < G'(k')$ with $\phi : G/P_0 \ra G'(k')/J$. In the archimedean case, by Corollary (3.2.18) of \cite{Z}, the group $J$ is an algebraic subgroup of the {\it real} algebraic group $G'$. The Zariski density of $\rho (\G)$ implies, by Corollary (3.2.19) of \cite{Z}, that the isotropy $J$ is compact. Since $G'(k')$ is non-compact, we finally get that $J$ is a proper algebraic group of $G'$.  The Zariski density of $\rho(\G)$ in $G'$ then implies that $\phi $ is not a constant map. \\

Consider the action of $G'$ by left translations on $G'/J$. The set $N$ of elements of $G'$ which acts trivially on $G'/J$ is a normal subgroup and is the intersection $\cap _{g\in G'} gJg^{-1}$ which is a proper algebraic group since $N< J$ and the latter, by the previous paragraph, is a proper algebraic subgroup. The simplicity of $G'$ then implies that $N$ is finite and central. Since $G'$ is assumed to be centreless, it follows that $N$ is trivial and hence that $G'$ acts faithfully on $G'/J$. \\

\subsubsection{A Lemma on Real Varieties.}  To proceed further in the proof in the archi\-me\-dean case, we need a preliminary result and to state it, we set up some notation. Suppose $X,Y, Z$ are smooth quasi-projective varieties defined over $\R$ and $\pi: Y \ra X$ a surjective morphism defined over $\R$, such that the map $\pi: Y(\R)\ra X(\R)$ (again denoted by $\pi$) is a surjective map of manifolds. Assume that $X(\R)$ and $Y(\R)$ are Zariski dense in $X$ and $Y$ respectively.  We fix $\sigma$-finite measures $\mu $ and $\nu $ on $X(\R)$ and $Y(\R)$ respectively, absolutely continuous with respect to the Lebesgue measure on each co-ordinate chart of $X(\R)$ (resp. $Y(\R)$).  Suppose $\phi: X(\R)\ra Z(\R)$ is a $\mu$ measurable map defined almost everywhere on $(X(\R), \mu)$ such that the composite map   $\phi \circ \pi: Y(\R)\ra Z(\R)$ coincides - almost everywhere with respect to $\nu$- with a rational map $\psi: Y \ra Z$ defined over $\R$. 
That is,  $\phi \circ \pi=\psi$ a.e. on $(Y(\R),\nu)$. 

\begin{lemma} \label{Rrational} With the foregoing assumptions, the map $\psi:Y(\R)\ra Z(\R)$ descends to an {\bf analytic} map defined on a Zariski open set $U$ of $X(\R)$ such that $\psi : U(\R)\ra Z(\R)$ and the map $\phi: X(\R) \ra Z(\R)$ coincide almost everywhere on $(X(\R), \mu)$ \rm{(}we abuse notation slightly and continue to denote the descended map on $U$ also by $\psi$\rm{)}. 
\end{lemma}

\begin{proof} ($1^o$) The zero set of a polynomial in the Euclidean space $\R^n$ has zero Lebesgue measure. This implies that if $U \subseteq X$ is a Zariski open set defined over $\R$, then $X(\R)\setminus U(\R)$ has zero Lebesgue measure and therefore, we may replace $X$ by $U$ in the lemma. \\

($2^o$) Suppose $Y \ra X$ is a finite map. By passing to a Zariski open subset $U\subseteq X$ defined over $\R$, we map assume that $Y \ra X$ is a finite cover and hence $Y(\R) \ra X(\R)$ is also a finite cover ({\it with possibly fewer sheets}). Under finite covers, it is immediate (by passing to evenly covered open sets) that a subset $E\subseteq (Y(\R),\nu)$ is measurable and has zero measure if and only if its image $\pi(E)\subseteq (X(\R),\mu)$ is measurable and has zero measure. \\

Let $V \subseteq X(\R)$ be an evenly covered co-ordinate open set in $X(\R)$ for the finite cover $Y(\R) \ra X(\R)$. Then $\pi ^{-1}(V)$ is a finite disjoint union of copies of $V$: $\pi ^{-1}(V) = \coprod _{i\in I} V_i$, with each $V_i$ open in $Y(\R)$ and $\pi: V_i \ra V$ an analytic isomorphism. Let $\sigma _i$ be the inverse map $\sigma_i: V \ra V_i$. This is also analytic. If $i, j\in I$, then the rational function $\psi : Y(\R) \ra Z(\R)$ is such that $\psi(\sigma _i(x))=\phi (x)=\psi (\sigma _j(x))$ for a conull set $E\subseteq V$. The analyticity of $\sigma _i$ and  $\sigma _j$ imply that $\psi(\sigma _i(x))=\psi (\sigma _j(x))$ for all $x\in V$. Hence $\psi$ descends to an analytic function $\psi$ on $X(\R)$ and $\phi =\psi$ a.e. on $(X(\R),\mu)$: for $x\in V$, define $\psi (x)= \psi(\sigma _i(x))$; then $\psi (x)=\phi (x)$ a.e. on $V$. \\

($3^o$) Suppose $Y=X\times W$ is a product of $\R$-varieties and $\pi : Y =X\times W \ra X$ the projection map. Then by assumption of the Lemma, there exists a co-null subset $E\subseteq Y(\R)= X(\R)\times W(\R)$ such that $\phi (\pi (x,w))= \psi (x,w)$ for all $(x,w)\in E$. 
Since $\phi \circ \pi (x,w)=\phi (x)$ is constant a.e. on $W(\R)$,  it follows, for all $x$ in some co-null subset $E\subseteq X(\R)$,  that the map $w\mapsto \psi(x,w)$ is constant a.e. in $W(\R)$. By the rationality of $\psi$ this means that $w\mapsto \psi (x,w)$ is constant {\it everywhere} on $W(\R)$. Thus the map $(x,w)\mapsto \psi (x,w)=\psi (x)$ is rational in $x$  on $X(\R)$. \\

The image of $E\subseteq Y(\R)=X(\R)\times W(\R)$ in $X(\R)$ may not be- a priori- measurable. However, the set $F=\{x \in X(\R): \phi(x)=\psi (x)\}$ is measurable and its inverse image $\pi ^{-1}(F)$ contains the set $E$. Since $E$ is co-null, it follows that $\pi ^{-1}(F)$ is co-null and hence $F$ is conull. That is, $\phi=\psi$ a.e. on $(X(\R), \mu)$. \\

($4^o$) In general, the map $Y \ra X$ (after possibly replacing $X$ by  a Zariski open set $U\subseteq X$ )   is a composite of two maps $p: Y \ra Y_1=X\times W$ and the projection $pr: Y_1 \ra X$, where $p$ is a finite cover. Since we have proved the lemma for $(Y_1,X)$ by ($3^o$) and for $(Y,Y_1)$ (by ($2^o$)), the lemma is proved in general. 
\end{proof}

\begin{remark}
 In an earlier version, we had wrongly asserted that the map $\psi$ extends to an algebraic map; this true for complex points, but for real points, we can only assert analyticity (as an example, we consider $X=Y=Z =\R^*$ , the map $Y\ra Z$ is identity and the map $\pi: Y\ra X$ is the map $x\mapsto x^3$; the map $\pi$ is 3 sheeted over complex points but one sheeted over real points. Then the map $X\ra Z$ i.e.  $\R ^* \ra \R^* $ is given by $ x \mapsto x^{1/3}$; this is analytic but not algebraic). The analyticity is sufficient for our purposes. \\

We had also assumed that the maps $Y \ra Z$ and $X\ra Z$ were \emph{everywhere defined}. But we actually have these maps defined only almost everywhere; therefore, we need to prove the lemma in the more general situation.\end{remark}

We recall that $G=G(\R)$ is a connected simple Lie group, $\G < G$ is a Zariski dense discrete subgroup and $H< G$ a semi-simple subgroup (we are denoting by $H$ the group of real points $H(\R)$). Let $\fg$ and $\fh$ be the Lie algebras of $G$ and $H$ respectively. Let $P_0$ be as before and $X=G/P_0$. Denote by $^{\g}(H)$ the conjugate $\g H \g ^{-1}$. We prove some preliminary results in preparation of the proof in the archimedean case. \\

By Proposition \ref{oneorbit}, there exists a measurable $\G$-equivariant map $\phi : G/P_0 \ra G'(k')/J$ where $J< G'(k')$ is a real algebraic subgroup. Moreover, if $\rho_i (h_i):=\rho (\g_i)\tilde \rho (h) \rho (\g_i)^{-1}$ with $H_i={\ }^{\g_i}(H)$, and if $k$ is as in Lemma \ref{G/Pgenerators}, then , by Lemma \ref{Hequivariance}, we have 
\[\phi (h_kh_{k-1}\cdots h_1x)=\rho _k(h_k)\cdots \rho _1(h_1)\phi (x) \quad a.e. \quad on \quad G/P_0.\]

\begin{lemma} \label{mapisnanlytic}
Under these assumptions, $\phi: G/P_0\ra Z$ coincides with an analytic function $\psi $ on $G/P_0$ almost everywhere on $G/P_0$. Moreover, the analytic function $\psi$ is defined everywhere on $G/P_0$.
\end{lemma}

\begin{proof} We have the equation 
\begin{equation} \label{*} \phi (h_k\cdots h_1x)=\rho _k(h_k)\cdots \rho _1(h_1)\phi (x)
\end{equation}
for all elements $(h_k, \cdots,h_1,x)$ in a conull set $E \subseteq H_k\times \cdots \times H_1 \times X$. By Fubini, there exists a conull set $X'\subseteq X$ such that for each $x\in X'$, the equation (\ref{*}) holds on a conull subset $E(x) \subseteq H_k\times \cdots \times H_1$. 
Write $Y=H_k\times \cdots \times H_1$. The map $\psi: Y \ra Z$ given by $(h_k, \cdots,h_1)\mapsto \rho _k(h_k)\cdots \rho _1(h_1)\phi (x)$ is rational for every $x\in X'$, and is a composite map of the form $\phi \circ \pi$ where $\pi :Y \ra X$ is the map $(h_k, \cdots, h_1)\mapsto h_k\cdots h_1 x$. By Lemma \ref{G/Pgenerators}, the map $\pi$ is surjective on real points. Consequently, by Lemma \ref{Rrational}, the map $\phi$ coincides with an analytic function $\psi $ a.e. on $X=G/P_0$, proving the first part of the claim. Moreover, by Lemma \ref{Rrational}, the map $\psi$ is defined on a Zariski open set $U$ in $G/P_0$. \\

There is a natural partial order on the set of pairs $(\psi , U)$ where $\psi : G/P_0 \ra Z$ is an analytic map defined over a Zariski open set $U$ and coinciding with the map $\phi$ almost everywhere on $G/P_0$: we say that $(\psi _1,U_1)\leq (\psi _2, U_2)$ if $\psi _2$ coincides with $\psi _1$ in $U_1$ and $U_1\subseteq U_2$. Since any increasing sequence of Zariski open sets terminates (by the Noetherean property for the Zariski topology), we may assume that our map $\psi$ is defined on a maximal Zariski open set $U$ (note that any Zariski open set is conull in $G/P_0$).\\

Given $h\in H_i$ for any $i$, $\rho _i (h)^{-1}\phi (h_ix)=\phi (x)$ a.e. on $X=G/P_0$. Then $\psi ' = \rho _i(h)^{-1}\psi (hx)$ equals $\psi (x)$ a.e.. But the  analytic functions $\psi '$ and $\psi $ are both defined on the intersection $V= U\cap h^{-1}(U)$ and are equal a.e. on $V$. Therefore, they coincide everywhere on $V$. Hence $\psi $ can be extended to an analytic function $\psi ''$ on the open set $U \cup h^{-1}(U)$. The maximality of $(\psi, U)$ then implies that $h^{-1}(U)=U$ for all $h\in H$. \\

The same is true for the $H_i$ for any $i \leq k$. Since the $H_i$ generate $G$, it follows that $g(U)=U$ for all $g\in G$; therefore, $U=G/P_0$ and $\psi $ is defined everywhere on $G/P_0$.  By replacing $\phi $ with $\psi$ (which equals $\phi$ a.e. on $G/P_0$) we may assume that the map $\phi :G/P_0 \ra G'/J$ is an {\it everywhere defined} analytic map, which is equivariant for the action of $H$ and of the discrete group $\G$. 
\end{proof}

\begin{theorem} \label{arch} Let $H$ be a semi-simple Lie subgroup (of
real rank at least two) of a simple Lie group $G$ which acts with
non-compact isotropies on $G/P_0$ (or, which satisfies the stronger
condition $\dim (H)> \dim(K)$ for a maximal compact subgroup $K$ of $G$).
Let $\Gamma <G$ be a Zariski dense discrete subgroup which
intersects $H$
in an irreducible lattice.  Let $\rho: \Gamma \ra G'(k')$ be a
homomorphism, with $k'$ an archimedean local field, and
$G'$ an absolutely simple algebraic group over $k'$. If $\rho (\Gamma
)$ is not relatively compact in $G'(k')$ and is Zariski dense in $G'$,
then $\rho $ extends to an algebraic homomorphism of $G$ into $R_{k'/\R}G'$
defined over $\R$.
\end{theorem}

\begin{proof} Here $G'(k')$ is viewed as the group of real points of a real algebraic group $R_{k'/\R}(G')$ where $R$ is the Weil restriction of scalars. We view the semi-simple linear group $G$ also as the group of real points of a real algebraic group.\\

If $\rho$ is an archimedean representation of $\Gamma $,
then, by Lemma \ref{mapisnanlytic} above, there exists a $\Gamma$-equivariant {\it everywhere defined analytic}  map
\[\phi: G/{P_0} \ra G'/J.\] 

The countable set  $\G$ may be written $\G=\{\g_1, \g_2, \cdots ,\g_m, \cdots \}$. Write $H_i$ for the conjugate $\g _i H \g _i^{-1}$. For a fixed $m$, denote by $Y_m$ the product $H_1\times \cdots\times H_m$. Then the set of elements in $Y_{m+1}$ whose last co-ordinate is the identity element is identified to $Y_m$ and denote by $Y$ the countable increasing  union $Y=\cup _{m=1}^{\infty} Y_m$. If $h\in Y$, then $h\in Y_m$ for some $m$; write $h=(h_1,h_{m-1}, \cdots,h_m)$,  and set $\Pi(h)$ to be the product (in $G$) $\Pi (h)=h_mh_{m-1}\cdots h_1$. We thus get a map $\Pi : Y \ra G$.   Since $G$ is the group generated by the subgroups $H_m$, it follows that the map $\Pi$ is surjective. \\

For $h\in Y_m$ with $h=(h_1, \cdots, h_m)$ as above, define the element  $R(h)$ as the product $R(h)=\rho _m(h_m)\cdots \rho _1(h_1)$ in the group $G'$.  Then $R: Y \ra G'$ is a set theoretic map. The equivariance of the map $\phi $ then says that $\phi (h_m\cdots h_1x)= \rho _m(h_m)\cdots \rho _1(h_1)\phi (x)$ for all $x\in G/P_0$ and for all   $h_i\in H_i$. Hence for all $h\in Y$ and all $x\in G/P_0$ we have $\phi (\Pi (h)x)=R(h)\phi (x)$.   Suppose $g=\Pi (h)=\Pi (h')$ for two elements $h,h' \in Y$ (we may assume that both $h,h' \in Y_m$ for some $m$). \\

 We then get 
\[\phi (gx)=R(h)\phi (x)=R(h')\phi (x) \quad \forall \quad x\in G/P_0.\]
For each $x\in G/P_0=X$, consider the conjugate $\phi (x)J\phi (x)^{-1}$ (the conjugate $yJy^{-1}$ depends only on the equivalence class $yJ$),  and consider the intersection $N=\cap _{x\in X} \phi (x)J\phi (x)^{-1}$. The $\G$-equivariance of the map $\phi$ shows that the algebraic group $N$ is normalised by $\rho (\G)$ and hence by the Zariski closure $G'$. The simplicity of $G'$ then implies that $N$ lies in the centre of $G'$ which by assumption, is trivial. 
The equation of the preceding paragraph then says that the element $R(h) ^{-1}R(h')$ lies in $N$ and is hence trivial. Therefore, the map $R:Y \ra G'$ is the same for two elements $h,h' \in Y$ with $\Pi (h)=\Pi (h')$. In other words $R$ descends to a map (which we still denote by $R$), with $R: G\ra G'$ such that for all $g\in G$ and all $x\in X$, $\phi (gx)=R(g)\phi (x)$. \\

Given $g_1,g_2 \in G$, we then get $R(g_1g_2)\phi (x)= \phi (g_1g_2x)=R(g_1)\phi (g_2x)= R(g_1)R(g_2)\phi (x)$ for all $x\in X$. The triviality of the group $N$ of the preceding paragraph then says that $R(g_1g_2)=R(g_1)R(g_2)$, and hence the set theoretic map $R:G \ra G'$ is an  abstract group homomorphism, with $\phi (gx)=R(g)\phi (x)$ for all $g\in G$ and $x\in X$. \\

The  intersection $1= N=\cap _{x\in X} \phi (x)J\phi (x)^{-1}$  of closed varieties is actually a finite intersection since the Zariski topology on $G'$ is Noetherian.  Therefore, there exist points $x_1, \cdots,  x_m \in X$ such that the intersection is 
$1= N=\cap _{i=1}^m \phi (x_i)J\phi (x_i)^{-1}$. Consider the $m$-tuple i.e. the point $p=(\phi (x_1), \cdots, \phi (x_m))\in G'/J \times \cdots \times G'/J$, the latter product is the $m$-fold product of $G/P_0$ with itself. The isotropy of $G'$ (under the diagonal action of $G'$ on $(G/J)^m$) at $p$ is the intersection of the groups $\phi (x_i)J\phi (x_i)^{-1}$ and is hence trivial. Thus the map $g'\mapsto g'(p)$  is an isomorphism from $G'$ onto its orbit $G'p$. \\

Since the map $\phi$ is analytic, the equality $R(g)p= (\phi (gx_1), \cdots, \phi (gx_m))$ for all $g\in G$ shows that the map $g\mapsto R(g)p$ is an analytic map from $G$ into the orbit $G'p$. Since the orbit is isomorphic to $G'$, we finally get that the abstract homomorphism $R: G\ra G'$ is an analytic homomorphism. But any analytic homomorphism of the algebraic group $G$ into the centreless  group $G'$ (i.e. $G'$ is an algebraic group such that  $G'(\C)$ has no centre), is algebraic. Hence the map $R$ is an algebraic homomorphism. The $\G$-equivariance of $\phi$ shows that $\rho (\g)=R(\g)$ for all $\g \in \G$ and hence $R$ extends $\rho$. This proves the Archimedean superrigidity.  
\end{proof}

Theorem \ref{theo:main} is an immediate consequence of Theorem \ref{nonarch} and Theorem \ref{arch}.
To prove Theorem \ref{ranktwo} we first observe:

\begin{lemma}\label{trick} Let $H$ be a semi-simple subgroup of simple
group $G$ and $K$ a maximal compact subgroup of $G$. Assume that
$\dim(H)>\dim(K)$. Let $G=KAN$ be an Iwasawa decomposition of $G$ and
$P_0=AN$. Then the isotropy subgroup of $H$ at any point in $G/P_0$ is
a non-compact subgroup of $H$.
\end{lemma}

\begin{proof} Since $G/P_0=K$, we have $\dim(G/P_0)=\dim(K)$, and since
$\dim(H)>\dim(K)$, at any point $p\in G/P_0$, the isotropy of $H$ is a
positive dimensional subgroup, which is conjugate to a subgroup of
$P_0$; the latter has no compact subgroups, hence the isotropy of $H$
at $p$ is a non-compact subgroup.
\end{proof}

Theorem \ref{ranktwo} is a particular case of Theorem  \ref{theo:main}, in view of Lemma \ref{trick}.

\section{Applications (Proof of Corollary \ref{SLn})}

Assume that $H=\SL_k(\R)$ and $G=\SL_n(\R)$, with $SL_k(\R)$ embedded in $SL_n(\R)$ in the top left hand corner. Under the assumptions of Corollary \ref{SLn}, we have $k >\frac{n}{\sqrt{2}}$ and $\dim (H)=k^2-1> \dim
(G/P_0)=n(n-1)/2$. Therefore, if $\Gamma$ is a Zariski dense discrete
subgroup of $G$ which intersects $H$ in a lattice, then by Theorem
\ref{ranktwo}, $\Gamma $ is super-rigid.  We now prove Corollary
\ref{SLn}.    \\

We now recall a result, which is a generalisation of Margulis' observation that  super-rigidity implies arithmeticity.  However, Margulis needed the discrete subgroup to be a lattice. We have not assumed that $\Gamma$ is a lattice (indeed, this is what is to be proved), and we also do not assume that $\Gamma $ is finitely generated.

\begin{theorem}[\cite{V3}]\label{arithmetic} Let  $G$ be  an absolutely simple real algebraic group and let $\Gamma $ be a
super-rigid discrete subgroup. Then there exists an arithmetic group $\Gamma _0$ of $G$ containing $\Gamma$.
\end{theorem}

Suppose $\Gamma _0 < G$ is arithmetic. This means that there exists a number field $F$ and a semi-simple linear algebraic $F$-group ${\bf G}$ such that the group ${\bf G}(\R\otimes F)$ is isomorphic to a product $\SL_n(\R)\times U$ of $\SL_n(\R)$ with a compact group $U$. Under this isomorphism, the projection of ${\bf G}(O_F)$ of the integral points of ${\bf G}$ into $G$ is commensurable with $\Gamma _0$.  The simplicity  of $\SL_n(\R)$ implies that ${\bf G}$ may be assumed to be absolutely simple over $F$. The group ${\bf G}$ is said to be an $F$-form of $\SL_n$.  Moreover, if $\Gamma _0$ contains
unipotent elements, then ${\bf G}$ cannot be anisotropic over $F$. Hence the $F$-rank of ${\bf G}$ is greater than zero. In that case, ${\bf G}(F\otimes \R)$ cannot contain compact factors (since compact groups cannot contain unipotent elements). This means that $F=\Q$.    \\

We now recall the classification of $\Q$-forms of $\SL_n$.   \\

[1] Let $d$ be a divisor of $n$ and $D$ a central division algebra
over $\Q$ of degree $d$. Write $n=md$. Then, the algebraic group ${\bf
G}=\SL_m(D)$ is a $\Q$-form of $\SL_n$. The rank of ${\bf G}$ is
$m-1=n/d-1$. If $d\geq 2$, then $m-1<n/2$.    \\

[2] Let $E/\Q$ be a quadratic extension and $D$ a central division
algebra over $E$ with an involution of the second kind with respect to
$E/Q$. Let $d$ be the degree of $D$ over $E$, suppose $d$ divides $n$
and let $md=n$. Let $h:D^m\times D^m\ra E$ be a Hermitian form with
respect to this involution, and let ${\bf G}=SU(h)$. Then, ${\bf G}$
is a $\Q$-form of $\SL_n$; its $\Q$-rank is not more than $m/2=n/2d\leq
n/2$.   \\

The classification of simple algebraic groups (see \cite{T}), implies
the following.

\begin{lemma} \label{classification} The only $\Q$-forms ${\bf G}$ of
$\SL_n$ are as above. In particular, if ${\bf G}$ is a $\Q$-form of
$\Q$-rank  strictly  greater than  $n/2$,  then  ${\bf G}$  is
$\Q$-isomorphic to $\SL_n$.
\end{lemma}

\begin{proof} [Proof of Corollary \ref{SLn}] By Theorem \ref{ranktwo}, the
group $\Gamma $ is super-rigid in $G$. By Theorem \ref{arithmetic},
$\Gamma $ is contained in an arithmetic subgroup $\Gamma _0$ of $G$.
Since $\Gamma _0 > \Gamma $ contains a finite index subgroup of
$\SL_k(\Z)$ by assumption, it follows that $\Gamma _0$ contains
unipotent elements. Therefore, the number field $F$ associated to
$\Gamma _0$ is $\Q$ and there is ${\bf G}$ a $\Q$-form of $\SL_n$ such that
$\Gamma_0$ is commensurate with ${\bf G}(\Z)$. Since $\Gamma<\Gamma_0$,
a finite index subgroup of $\SL_k(\Z)$ is a subgroup of ${\bf G}(\Q)$ and hence its
Zariski closure $\SL_k$ is a $\Q$-subgroup of ${\bf G}$. Hence the $\Q$-rank of ${\bf G}$ is
not less than the $\Q$-rank of $\SL_k$ which is $k-1>n/2$ by assumption.    \\

By Lemma \ref{classification}, the $\Q$-form ${\bf G}$ is isomorphic
to $\SL_n$.  Hence $\Gamma _0$ is commensurable with $\SL_n(\Z)$.
Moreover, the $\Q$-inclusion of ${\bf H}=\SL_k$ in ${\bf G}=\SL_n$ is
the  standard one  described before  the statement  of Theorem
\ref{spllinear}.    \\ 

Now, $\Gamma$ is Zariski dense and contains a finite index subgroup of $\SL_k(\Z)$.  Let
$e_1,e_2,\cdots, e_n$ be the standard basis of $\Q^n$.  Consider the
change of basis which interchanges $e_k$ and $e_n$ and all other
$e_i$'s are left unchanged. After this change of basis, (which leaves
the diagonal torus stable), the group $\Gamma$ (or rather, a
conjugate of it by the matrix effecting this change of basis) contains
the highest root group and the second highest root group (in the usual
notation for $\SL_n$ the positive roots occur in the Lie algebra of
upper triangular matrices). By Theorem (3.5) (or Corollary (3.6)) of
\cite{V1}, $\Gamma $ must be of finite index in $\SL_n(\Z)$. This
proves Corollary \ref{SLn}.   \\ 
\end{proof}

Corollary \ref{Spn} is proved in an analogous way.

\section{Proof of Theorem \ref{spllinear}}

\begin{notation} Let $k\geq 3$ and $n\geq k+2$ be integers. The standard $n$ dimensional real vector space is denoted $\R ^n$ and its standard basis is denoted $e_1,e_2, \cdots,e_n$. Write $W=\sum _{i\leq k} \R e_i$ and $W'=\sum _{i>k}\R e_i$. Then $\R^n$ is the direct sum $\R^n=W\oplus W'$. The group $SL(W)$ is viewed as the subgroup of $SL_n(\R)$ which acts via the standard representation on $W$ and acts trivially on $W'$. Then the set $(\R ^n)^W$ of $SL(W)$-invariant vectors  in $\R^n$ is precisely $W'$, and $SL_k(\R)=SL(W)<SL_n(\R)$ is the "top left hand corner". \\

We also write $(\R e_1)' =\sum _{i>1} \R e_i$ and $(\R e_k)' =\sum _{i\neq k} \R e_i$. Then $W'\subseteq (\R e_1)'$ and $W'\subseteq (\R e_k)'$ and we have the decomposition 
\begin{equation} \label{decomposition}
\R ^n= \R e_1 \oplus (\R e_1)'= \R e_k\oplus (\R e_k)' = W \oplus W'.
\end{equation} 
Given $g\in SL_n(\R)$, we have $g(e_1)=(w(g),w'(g))$ according to the decomposition $\R ^n=W\oplus W'$. 

Fix an integer $m\geq 1$. If $i,j\leq n$ denote by $E_{ij}$ the $n\times n$ matrix in $M_n(\R)$ whose $ij$-th entry is $1$ and all other entries are zero. Let $u_0=1+E_{1k}$; then $u_0^m=1+mE_{1k}$ and $u_0\in SL(W)\cap SL_n(\Z) $. Moreover, the kernel of $(u_0^m-1)$ is $\sum _{i\neq k} \R e_i= (\R e_k)'$, and the image of $(u_0^m-1)$ is $\R e_1$. Given $g\in SL_n(\R)$, write $u(g)=gu_0^mg^{-1}$. \\
\end{notation}

\begin{lemma} \label{zariskiopenU} With the preceding notation let 
\[U=\{ g\in SL_n(\R): g(e_1) \notin W \cup W',\ g^{-1}(W) \not \subseteq (\R e_k)',  g^{-1}(W') \not \subseteq (\R e_k)' \quad \]
\[and \quad (u(g)-1)w'(g)\neq 0 \}.  \] Then $U$ is a \emph{nonempty} Zariski open set in $SL_n(\R)$.
\end{lemma}

\begin{proof} Suppose  $X,Y \subseteq \R ^n$ are two \emph{proper non-zero} subspaces. The set $\{g\in SL_n(\R): g(X) \not \subseteq  Y\}$ is Zariski open and nonempty: indeed, if $g(X)\subseteq Y$ for all $g\in SL_n(\R)$, then the proper subspace  $Y$ contains the span $\sum _{g\in G} g(X)$ which is a non-zero $SL_n(\R)$-invariant subspace, contradicting the irreducibility of the action of $SL_n(\R)$ on $\R^n$. Since a finite intersection of non-empty Zariski open sets in $\R^n$ is also non-empty and Zariski open, it follows that the set 
\[V =\{g\in SL_n(\R): g(\R e_i) \not \subseteq W, g(\R e_1)\not \subseteq W'  \quad and \]
\[g^{-1}(W')\not \subseteq (\R e_k )' \quad g^{-1}(W') \not \subseteq (\R e_k)'   \}  \] is non-empty and Zariski open. \\

Suppose $u(g)=gu_0^mg^{-1}$ and $g(e_1)=(w(g),w'(g))$ as before such that $(u(g)-1)w'(g)$ is identically zero on the Zariski open set $V$; since this function extends to all of $G$, this is identically zero on $G$ as well. Fix $x$ in the open set $V$. Then $(u(x)-1)w'(x)=0$. That is, $(u_0^m-1)x^{-1}w'(x)=0$. The kernel of $u_0^m-1$ is $(\R e_k)'$, and hence $x^{-1}w'(x) \in (\R e_k)'$ for all $x\in V$. Moreover, since $x(e_1)\notin W\cup W'$, it follows that under the decomposition $\R ^n=W\oplus W'$, $x(e_1)=(w(x),w'(x))$ and $w(x)\neq 0$, $w'(x)\neq 0$.  \\

We have the decomposition $\R ^n= \R e_1\oplus (\R e_1)'$; accordingly, we may write $x^{-1}w'(x)=\l e_1+\xi$, $\xi \in (\R e_1)'$ and $\l \in \R$. If $\xi=0$, then $w'(x)=\l x(e_1) = \l (w(x), w'(x))$ which shows that $w(x)=0$; this is impossible since $x\in V$. Hence  $\xi \neq 0$. \\

We now observe that the function $g\mapsto g(e_1)$ descends to a function on the quotient $G/M$ where $M=\{ a= \begin{pmatrix} 1 & 0 \\ 0 & a \end{pmatrix}: a \in SL_{n-1}(\R)\}$ is the subgroup $SL((\R e_1)')$. Since $M$ acts transitively on $(\R e_1)' \setminus \{0\}$, it follows that there exists $a\in M$ with $a^{-1}(\xi)=e_k$. Put $g=xa$. Then $w(g)=w(x), \quad w'(g)=w'(x)$, and \[g^{-1}w'(g)= a^{-1}x^{-1}w'(x)=a^{-1}(\l e_1+ \xi)=\l e_1+a^{-1}\xi=\l e_1+e_k.\] 
We get $(u_0^m-1)g^{-1}w'(g)= (u_0^m-1)(e_k)=me_1\neq 0$ since $u_0$ fixes the vector $e_1$. Hence, multiplying by $g$ on the left, we get  $(u(g)-1)w'(g) =mg(e_1) \neq 0$. Hence the set $\{ g\in SL_n(\R): (u(g)-1)w'(g)\neq 0\}$ is non-empty and Zariski open. The intersection of this set with $V$ is the set $U$ of the lemma and $U$ is therefore non-empty and open.  
\end{proof}

\begin{lemma}\label{SL(k+1)} Suppose $m\geq 1$ and $W, u_0,g\in U$ are as in Lemma \ref{zariskiopenU} and $u(g)=gu_0^mg^{-1}$. Let $\mathcal G$ be the Zariski closure in $SL_n(\R)$ of the group generated by $SL(W)$ and $u(g)$. Then there is an element of $SL_n(\R)$ which conjugates $\mathcal G$ \emph{isomorphically onto} the top left hand corner $SL_{k+1}(\R)$, and which is identity on $SL_k(\R)$. 
\end{lemma}

\begin{proof} The element $u_0$ is such that $u_0^m -1$ maps all of $\R ^n$ onto $\R e_1$; hence $u(g)-1$ maps all of $\R ^n$ onto $\R g(e_1)$. A subspace of $\R^n$ which contains $g(e_1)$ is therefore stable under $u(g)-1$ and hence under $u(g)$. Similarly, every $h\in SL_k(\R)$ is such that $h-1$ maps all of $\R ^n$ into $W=\R ^k$; therefore, a subspace of $\R ^n$ which contains $W$ is invariant under $h-1$ and hence under all $SL_k(\R)$. This implies that $F=F(g)= W+ \R g(e_1)$ is stable both under $SL(W)=SL_k(\R)$ and under $u(g)$; thus $F$ is $\mathcal G$-stable.  \\

$(1^o)$ We first show that $\mathcal G$ acts irreducibly on $F$. Suppose $E\subseteq F$ is a non-zero $\mathcal G$-invariant subspace. \\

Case 1: The intersection $E\cap W$ is non-zero. Since both $E$ and $W$ are $SL(W)$-stable, and $SL(W)$ acts irreducibly on $W$, it follows that $E\cap W=W$ and hence $E\supseteq W$. \\

Since $g\in U$, $g^{-1}(W)$ is not contained in $(\R e_k)' =\ker (u_0^m-1)$, and there exists a vector $w_1\in W\subseteq E$, such that $(u_0^m-1)g^{-1}(w_1) =\l e_1$ for some non-zero scalar $\l$. Therefore, $\l g(e_1)=(u(g)-1)w_1$ lies in $E$ since $E$ is stable under $u(g)$. Hence $E\supseteq \R g(e_1)$ as well and so $E=F$. \\

Case 2: The intersection $E\cap W=0$ (we show that it is not possible). Then $E$ maps injectively into the one dimensional quotient $F/W=\R g(e_1)$.  Hence $E=\R e$ is one dimensional, and is $SL(W)$ stable; hence $SL(W)$ acts trivially on $E$ and $e\in W'$. Since $U$ is unipotent and $E$ is $u(g)$-stable, it follows that $e$ is an eigenvector with eigenvalue $1$ for the action of $u(g)$, meaning that $(u(g)-1)e=0$. \\

After replacing $e$ by a scalar if necessary, we map assume that $e$ maps to $g(e_1)$ in $F/W$; that is, $g(e_1)=e+w$ for some $w\in W$. In other words, $g(e_1)=(w,e)$. But $g(e_1)=(w(g),w'(g))$ in $\R ^n=W\oplus W'$. Hence $e=w'(g)$, and $( u(g)-1)w'(g)=0$ by the preceding paragraph. This contradicts the assumption that $g$ lies in the open set $U$. Hence the case $W\cap E=0$ cannot arise. \\

$(2^o)$ Since $SL(W)$ is generated by unipotent elements and $u(g)$ is unipotent, it follows that the Zariski closure $\mathcal G$ is connected. Therefore, its Lie algebra $\fg$ also acts irreducibly on $F=W+\R g(e_1)= W\oplus \R w'(g)=\R^{k+1}$, so is contained in $\mathfrak {sl}_{k+1}$ and contains $\mathfrak {sl}_k$. We now show that any such Lie algebra $\fg$ must be all of $\mathfrak {sl}_{k+1}$, assuming $k\geq 3$.\\

To see this, suppose $k\geq 3$ and decompose $\mathfrak{sl}_{k+1}$ as a module over $\mathfrak{sl}_k$ (sitting in the top left hand corner). Write a matrix $X\in \mathfrak{sl}_{k+1}$ in block form $X=\begin{pmatrix} A & B \\ C & d\end{pmatrix}$, where $A\in \mathfrak{gl}_k$, $B\in \R ^k$ (viewed as column vectors of size $k$), $C\in (\R ^k)^*$ (viewed as row vectors of size $k$) and $d$ a scalar such that $d+trace(A)=0$. The maps $X\mapsto B$ and $X\mapsto C$ take $\mathfrak{sl}_{k+1}$ into $\R ^k$ and $(\R ^k)^*$ respectively and are module maps of $\mathfrak{sl}_k$. (Similarly $A$). We therefore get a decomposition of $\mathfrak{sl}_{k+1}$ as a module over $\mathfrak{sl}_k$ ($triv$ is the trivial one dimensional representation of $\mathfrak{sl}_k$):
\[ \mathfrak{sl}_{k+1}= \mathfrak{sl}_k \oplus \R ^k \oplus (\R ^k)^* \oplus triv.\]
Since $k\geq 3$, the irreducible $\mathfrak{sl}_k$-modules $\R ^k$ and $(\R ^k)^*$ are not isomorphic, and hence in the above
decomposition, each irreducible module occurs only once. Consequently, if $M\subseteq \mathfrak{sl}_{k+1}$ is a submodule for $\mathfrak{sl}_k$ and maps non-trivially into $\R ^k$ or $(\R ^k)^*$, then it contains $\R ^k$ or $(\R ^k)^*$. We apply this observation to the submodule $\fg$. If the map $B: \fg \ra \R ^k$ is identically zero, then $\fg$ is contained in the subalgebra $\{ 
\begin{pmatrix} A & 0 \\ C &d\end{pmatrix}: X\in \mathfrak{sl}_{k+1}\}$, which shows that the line $\R e_{k+1}$ is $\fg$-stable. This contradicts the irreducibility of the action of $\fg$ on $\R ^{k+1}$, and hence $B$ is not identically zero on $\fg$; that is, $\fg \supseteq \R ^k$. \\

Similarly, if the map $C: \fg \ra (\R ^k)^*$ is identically zero, then $\fg \subseteq\{ X\in \mathfrak{sl}_{k+1}: X=\begin{pmatrix} A & B \\ 0 & d \end{pmatrix}\}$ leaves $\R ^k$ stable, contradicting the irreducibility of $\fg$ on $\R ^{k+1}$. Therefore, $\fg \supseteq (\R ^k)^*$. The bracket of $\R ^k$ and $(\R ^k)^*$ maps onto the trivial representation and hence $\fg \supseteq triv$ as well. \\

By assumption, $\fg \supseteq \mathfrak{sl}_k$.  Hence $\fg=\mathfrak{sl}_{k+1}$. Therefore, the image of the group $\mathcal G$ 
in $Aut (W+\R g(e_1))$ is $SL_{k+1}(\R) \subseteq SL_n(\R)$ after a conjugation. \\

($3^o$) We know from the preceding that $(u(g)-1)w'(g)\neq 0$. The kernel of $u(g)-1$ on $\R^n$ is $g(\R e_k)'$ and has co-dimension one in $\R ^n$; hence its intersection with $W'$ has codimension one in $W'$ (since, by assumption $g\in U$ - cf Lemma \ref{zariskiopenU}- $W'$ is not contained in $g(\R e_k)'$). Therefore, $W^{\prime \prime}=W'\cap\ker (u(g)-1)$ has co-dimension one in $W'$ and and $\R^n= W\oplus \R w'(g) \oplus W^{\prime \prime}$; The action of $SL(W)$ on $W^{\prime \prime}$ is trivial since $W^{\prime \prime}\subseteq W'$; the action of $u$ on $W^{\prime \prime}$ is trivial since $W^{\prime \prime} \subseteq\ker(u(g)-1)$. Therefore, $\R ^n=(W+\R g(e_1))\oplus W ^{\prime \prime}$ is a decomposition of $\mathcal G$-modules and hence $\mathcal G$, after a change of basis (i.e. after a conjugation),  is the top left hand corner $SL_{k+1}(\R)$ in $SL_n(\R)$, proving the lemma.  
\end{proof}

We are now ready to complete the proof of Theorem \ref{spllinear}.

\begin{proof}[Proof of Theorem \ref{spllinear}] We have
$\SL_3(\Z)$ is virtually contained (in the top left corner) in the
Zariski dense discrete subgroup $\Gamma $ of $\SL_n(\R)$.  We will
prove by induction, that for {\bf every} $k\geq 3$ with $k\leq n$, a
conjugate of $\SL_k(\Z)$ is virtually a subgroup of $\Gamma$.
Applying this to $k=n$ gives us Theorem \ref{spllinear}.   \\

Suppose that for some $k \geq 3$, $\SL_k(\Z)$ is virtually contained
in $\Gamma$. Let $\gamma\in \Gamma$, and let $\Delta =\Delta (\gamma)$
denote the subgroup of $\Gamma $ generated by a finite index subgroup of $\SL_k(\Z)$ and a
conjugate of the unipotent element $\gamma(1+me_{1k})\gamma ^{-1}$.
Assume further, that the element $\gamma $ is in the open set
$U$ of Lemma \ref{zariskiopenU}. By Lemma \ref{SL(k+1)},
the Zariski closure of the group $\Delta $ maps isomorphically, under a conjugation,  onto $\SL_{k+1}(\R)$. \\

But $\D$ is a Zariski dense discrete subgroup of $SL_{k+1}(\R)$ which intersects the top left hand corner $SL_k(\Z)$ in a subgroup of finite index. By Corollary \ref{SLn} it follows that after a conjugation, $\D$ intersects $SL_{k+1}(\Z)$ in a subgroup of finite index.  
We have thus
proved that if a subgroup of finite index in $\SL_{k} (\Z)$ is contained in $\Gamma $, then, after replacing $\Gamma$ by a conjugation if necessary,  a subgroup of finite index in $\SL_{k+1}(\Z)$ is contained in $\Gamma $, provided $k+1\leq n$. Thus the induction is completed and
therefore Theorem \ref{spllinear} is established.
\end{proof}

\section{The rank one case}

In this last section, we will see that the situation for Nori's
question is completely different in real rank one. More precisely, we
have the following result.

\begin{theorem}\label{rankone} Let $G$ be a real simple Lie group of
real rank one and $H\subseteq G$ a non-compact semi-simple subgroup.
Suppose that $\Delta $ is a lattice in $H$.  Then, there exists a
Zariski dense discrete subgroup $\Gamma $ in $G$ of infinite co-volume
whose intersection with $H$ is a subgroup of finite index in the
lattice $\Delta$.
\end{theorem}

Suppose that $H$ is a simple non-compact subgroup of a simple group
$G$ of real rank one. Let $P$ be a minimal parabolic subgroup of $G$
which intersects $H$ in a minimal parabolic subgroup $Q$. The group
$G$ acts on $G/P$ and $H$ leaves the open set ${\mathcal U}=(G/P)
\setminus (H/Q)$ stable. Let $\Delta $ be a discrete subgroup of $H$.

\begin{lemma} \label{pingpong} Given compact subsets $\Omega _1$ and
$\Omega _2$ in the open set ${\mathcal U}$, the set ${\mathcal U}_H =
\{h\in H: h\Omega _1\subseteq\Omega _2\}$ is a compact subset of $H$.
Further, the group $\Delta  $ acts properly discontinuously on
${\mathcal U}$.
\end{lemma}

\begin{proof} Choose, as one may, a maximal compact subgroup $K$ of
$G$ such that $K\cap H$ is maximal compact in $H$. The set ${\mathcal
U}$ is invariant under $H$ and hence under $K\cap H$. We may assume
that the compact sets $\Omega _1$ and $\Omega _2$ are invariant under
$K\cap H$ of $H$. The Cartan decomposition of $H$ says that we may
write $H=(K\cap H)A^+(K\cap H)$ with $A$ a maximal real split torus
(of dimension one) and $A^+ =\{ a\in A: 0<\alpha (h)\leq 1\}$ where
$\alpha$ is a positive root of $A$. Let $W$ denote the normaliser
modulo the centraliser of $A$ in $G$.  This is the relative Weil
group. Since $\R$-rank$(G)=1$, it follows that $W=\{1,\kappa\}$ has
only two elements. Since $\R$-rank$(H)=1$, it follows that $\kappa
\in H$. We have the Bruhat decomposition $G=P\cup U\kappa P$, where where $\kappa $ is the non-trivial element of the Weil group of $A$
in $G$ and $U$ is the unipotent radical of the minimal parabolic subgroup $P$.    \par 
If possible, let $h_m$ be a sequence in ${\mathcal U}_H$ which tends
to  infinity.  It  follows  from the  previous paragraph  that
$h_m=k_ma_mk_m'$ with $k_m,k'_m\in K\cap H$ and $a_m\in A^+$, and
$\alpha (a_m)\ra 0$ as $m\ra +\infty$. Let $p$ be an element of
$\Omega _1 \subseteq {\mathcal U}$. By the Bruhat decomposition in $G$
($G$ has real rank one), it follows that ${\mathcal U}\subseteq U\kappa
P$. We may write $p=ukP$ for some $u\in U$. Since $h_mp\in
\Omega _2$ and the latter is compact, we may replace $h_m$ by a
sub-sequence and assume that $h_mp$ converges, say to $q\in \Omega
_2$. Since $\Omega _2$ is stable under $K\cap H$, we may assume that
$k_m=1$, by replacing $q$ by a sub-sequential limit of $k_m^{-1}q$.    \\

We compute $h_mp=a_mk'_mu\kappa P \in \Omega _2 \subseteq{\mathcal U}$.
The convergence of $k_m'$ says that $k'_mu\kappa P=u_m\kappa P$ with
$u_m$ convergent (possibly after passing to a sub-sequence). We write
$^a(u)\stackrel{defn}{=}aua^{-1}$.  Then, $h_mp=^{a_m}(u_m)\kappa P$
since $\kappa $ conjugates $A$ into $P$ (in fact into $A$). Since
$\alpha (a_m)\ra 0$, and $u_m$ are bounded, conjugation by $a_m$
contracts $u_m$ into the identity and $h_mp$ therefore tends to
$\kappa P$. But the latter is in $H/Q$ since $\kappa $ already lies
in $H$.  Hence the limit does not lie in $\Omega _2$, contradicting
the fact that $\Omega _2$ is compact. Therefore ${\mathcal U}_H$ is
compact.    \\

The second  part of the  lemma immediately follows  since the
intersection of $\Delta $ with ${\mathcal U}_H$ is finite.
\end{proof}

\begin{lemma} \label{goodset} Let $\Gamma \subseteq G$ be a Zariski dense
discrete subgroup. There exists an element $\gamma \in \Gamma $ such
that the $\gamma $ translate of $HP/P$ does not intersect $HP/P$:
\[\gamma (HP/P) \cap HP/P=\emptyset.\]
\end{lemma}

\begin{proof} First, suppose that $V$ is a compact set contained in
$G/P\setminus \{P,\kappa P \}$ (with $\kappa $ as in the proof of
Lemma \ref{pingpong}). By Bruhat decomposition, $V\subseteq U\kappa P/P$
and its $U$ part lies in a compact set. After a conjugation, we assume
as we may, that $\Gamma $ contains a semisimple element $t$ in $A$ such that
the positive powers of $Ad (t)$ contract elements of $U$ into identity.
Moreover, this contraction is uniform on a compact subset of $U$. Hence
there exists a positive power $t^m$ of $t$ such that
$t^m (V)$ lies in an arbitrarily small neighbourhood of $\kappa P \in G/P$.    \\

The group $\Gamma $ is Zariski dense in $G$ and $HP/P$ is a Zariski closed subset
of $G/P$. Therefore, there exists $g\in G$ such that $g^{\pm 1}P \notin HP/P$ and
$g^{\pm 1}\kappa P\notin HP/P$. Now, $HP/P\subseteq G/P\setminus \{gP, g\kappa P\}$ is a
compact set. Hence $V=g^{-1}(HP/P)$ is a compact set in $G/P\setminus \{P,\kappa P\}$.
Applying a large positive power of $t\in \Gamma \cap A$ as in the preceding paragraph,
we see that for some large integer $m$, $t^m V$ lies in a small neighbourhood of
$\{P,\kappa P\}$. Therefore, $gt^m g^{-1}(HP/P)$ lies in a small neighbourhood of
$\{gP,g\kappa P\}$.    \\

By the choice of the element $g\in \Gamma $, the latter
set does not intersect $HP/P$. Choose a small neighbourhood of
$\{gP, g\kappa P\}$ which does not intersect $HP/P$; if $m$ is large, then
the set $gt^mg^{-1}(HP/P)$ lies in this neighbourhood, and hence does not intersect
$HP/P$. We take $\gamma =gt^mg^{-1} \in \Gamma$. This proves the Lemma.
\end{proof}

Given a lattice $\Delta$ in $H$, and given a compact subset $\Omega
\subseteq{\mathcal U}$, Lemma \ref{pingpong} shows that there exists a
finite index subgroup $\Delta '$ of $\Delta $ such that non-trivial
elements of $\Delta '$ drag $\Omega $ into an arbitrarily small
compact neighbourhood ${\mathcal V}$ of $H/Q=HP/P$.   \\

By  Lemma \ref{goodset}  there exists  $g\in G-HP$  such that
$g({\mathcal V})\subseteq{\mathcal U}$.  Replacing $H/Q$ by $g(H/Q)$
and $\Delta $ by $g\Delta g^{-1}$, we see that all points of $H/Q$ are
dragged, by non-trivial elements of $g\Delta g^{-1}$, into a small
neighbourhood of $H/Q$.  The ping-pong lemma then guarantees that
there exists a finite index subgroup subgroup $\Delta ''$ such that
the group $\Gamma $ generated by $\Delta ''$ and $g\Delta ''g^{-1}$ is
the free product of $\Delta''$ and $g\Delta ''g^{-1}$.   \\

We may replace $g$ by a finite set $g_1,\cdots, g_k$ such that for each
pair  $i,j$,  the  intersections  $g_i(H/Q)\cap  g_j(H/Q)$  and
$g_i(H/Q)\cap H/Q$ are all empty.  By arguments similar to the
preceding paragraph, we can find a finite index subgroup $\Delta _0$
of $\Delta $ such that the group $\Gamma $ generated by $g_i \Delta _0
g_i ^{-1}$ is the free product of the groups $g_i\Delta _0 g_i^{-1}$,
and by choosing the $g_i$ suitably, we ensure that $\Gamma $ is
Zariski dense in $G$. This proves Theorem \ref{rankone},
since $\Gamma $ is
discrete and since it operates properly discontinuously on some open
set in $G/P$, $\Gamma $ cannot be a lattice. This proves Theorem
\ref{rankone}.   \\

\subsection*{Acknowledgments}  T.N.Venkataramana extends to Madhav Nori his hearty thanks for
many very helpful discussions over the years; the present paper
addresses a very special case of a general question raised by him
about discrete subgroups containing higher rank lattices. He also
thanks Marc Burger for explaining the structure of isotropy subgroups
of measures on generalized Grassmanians. The excellent hospitality of
the Forschunginstitut f\"ur Mathematik, ETH, Zurich while this paper
was written is gratefully acknowledged.   \\

The paper was completed during Indira Chatterji's visit to the Tata
Institute and she thanks Tata Institute for providing fantastic
working conditions.   \\

The authors thank Yves Benoist for the reference \cite{J-M}, Gregory Margulis for interesting remarks, Yehuda Shalom for pointing out to us that in Theorem \ref{theo:main}, we may
replace the hypothesis that $G$ is simple, by the assumption that $G$ is semi-simple. We also warmly thank anonymous referees for many useful comments.\\


\begin{thebibliography}{JPSH}
\bibitem[B]{B} H. Bass, {\it Representation rigid subgroups}, Pacific
Journal on Math. {\bf 86} (1980), 15-51.   \\ 
\bibitem[B-L]{B-L}  H. Bass  and  A. Lubotzky,  {\it  Nonarithmetic
super-rigid groups: counterexamples to Platonov's conjecture}, Ann. of
Math.(2) {\bf 151} (2000), {\bf 3}, 1151-1173. \\ 
\bibitem[BMS]{BMS} H. Bass, J. Milnor and J.-P. Serre, {\it Solution
of  the congruence  subgroup problem  for $\SL_n(n\geq  3)$ and
$\Sp_{2n}(n\geq 2)$}, Publ. Math. I.H.E.S., {\bf 33} (1967), 59--137. \\ 
\bibitem[Ben1]{Ben1}Y. Benoist, {\it Automorphismes de cones convexes}, Invent. Math. {\bf 141} (2000) p.149--193. \\ 
\bibitem[Ben2]{Ben2}Y. Benoist, {\it Convexes Divisibles I},
Algebraic groups and arithmetic, Tata Inst. Fund. Res. Stud. Math. {\bf 17} (2004) p.339--374. \\ 
\bibitem[BHC]{BHC}A. Borel and Harish-Chandra, {\it Arithmetic subgroups of algebraic groups}, Bull. Amer. Math. Soc. Volume 67, Number 6 (1961), 579--583. \\ 
\bibitem[DGK]{DGK}
J. Danciger, F. Guéritaud, F. Kassel, {\it Combination theorems in convex projective geometry}, preprint 2024 https://arxiv.org/abs/2407.09439\\
\bibitem[F-K]{F-K} R. Fricke and F. Klein, {\it Vorlesungen Uber die
theorie der automorphische funktionen}, Johnson Reprint Corp., New
York; B.G.Teubner Verlagsgesselschaft, Stuttgart 1965 Band I und II. \\ 
\bibitem[J-M]{J-M} D. Johnson and J. Millson, {\it Deformation spaces associated to compact hyperbolic manifolds},
Progress in math {\bf 67} (1984), 48--106. \\ 
\bibitem[H]{Hal} P. Halmos, {\it Measure Theory}, GTM volume 68, ISBM: 978-1-4684-9440-2 Springer.\\
\bibitem[M] {M} G. A. Margulis, {\it Discrete subgroups of semi-simple
lie groups}, Berlin-Heidelberg-New York, Springer-Verlag, (1991).  \\ 
\bibitem[O]{O} Hee Oh, {\it Discrete groups generated by lattices in
opposite horospherical subgroups}, J. Algebra {\bf 203} (1998),
No.{\bf 2}, 621-676.  \\ 
\bibitem[P-R]{P-R} V. P. Platonov and A. Rapinchuk, {\it Algebraic Groups
and Number Theory}, Academic Press, (1994).  \\ 
\bibitem[R]{R} M. S. Raghunathan,  {\it A Note  on Generators for
arithmetic subgroups of algebraic groups} Pac J. Math.  {\bf 152}
(1992), 365-373.  \\ 
\bibitem[R2]{R2} M.S.Raghunathan, Cohomology of Arithmetic Subgroups of
Algebraic Groups (I) and (II). Ann of Math, {\bf 86}, (1967) 409-424,
{\bf 87}, (1967), 279-304. \\ 
\bibitem[T]{T} J. Tits, {\it Classification of Algebraic Semi-simple
Groups}, Algebraic Groups and Discontinuous Groups, Proc. Symp. Pure
Math., Amer. Math. Soc. Providence, Rhode Island, Boulder, Colorado,
(1965), 33-62.  \\ 
\bibitem[V1]{V1} T. N. Venkataramana, {\it Zariski Dense subgroups of
arithmetic groups}, J. of Algebra {\bf 108}, (1987), 325-339. \\ 
\bibitem[V2]{V2} T. N. Venkataramana, {\it On some rigid subgroups of
semi-simple Lie groups}, Israel J. Math., {\bf 89} (1995),
no.{\bf 1-3}, 227-236.  \\ 
\bibitem[V3]{V3} T. N. Venkataramana, {\it On the arithmeticity of
certain rigid subgroups} CR.Acad. Aci. Paris. s'er I, Math {\bf 316}
(1993), no. {\bf 4}, 321-326.  \\ 
\bibitem[Z]{Z} R. Zimmer, {\it Ergodic Theory and Semi-simple groups},
Boston-Basel-Stuttgart, Birkhauser, (1984).  \\ 
\end{thebibliography}
\end{document}